\documentclass[12pt]{amsart}
\usepackage{amsmath,amssymb,amscd,amsthm,mathrsfs}

\theoremstyle{definition}
\newtheorem{thm}{Theorem}
\newtheorem{prop}[thm]{Proposition}
\newtheorem{lem}[thm]{Lemma}
\newtheorem{cor}[thm]{Corollary}
\newtheorem{defi}[thm]{Definition}

\newtheorem{prob}[thm]{Problem}

\newtheorem{rem}[thm]{Remark}

\begin{document}

\title{The heat kernel on $SL(2,\mathbb{R})$}

\author[Shota MORI]{Shota MORI}
\address[Shota MORI]{Graduate School of Mathematics, Nagoya University, Furocho, Chikusaku, Nagoya, 464-8602, Japan}
\email{mori.shouta@i.mbox.nagoya-u.ac.jp}

\date{}


\begin{abstract}
Let $G$ be a noncompact semisimple Lie group equipped with a certain invariant Riemannian metric. Then, we can consider a heat kernel function on $G$ associated to the Riemannian metric. We give an explicit formula for the heat kernel when $G=SL(2,\mathbb{R})$. The main tools are spherical transform and Helgason-Fourier transform for homogeneous vector bundles studied by R. Camporesi.
\end{abstract}

\maketitle

\section{Introduction}
The main purpose of this paper is to calculate a heat kernel on $SL(2,\mathbb{R})$ explicitly, where we treat $SL(2,\mathbb{R})$ as a Riemannian manifold equipped with a certain invariant metric. The main theorem (Theorem \ref{MyThm1}) is stated at the end of Section 7.

First of all, we present the significance of our problelm as well as a background history. In 1961, V. Bargmann(\cite{b1}) introduced a Hilbert space
\begin{align*}
\mathfrak{F}_{n,t}=\{F:\mathbb{C}^n\rightarrow\mathbb{C}:\mathrm{holomorphic}\, |\, \int_{\mathbb{C}^n}|F(z)|^2\rho_{\mathbb{C}^n}(t,z)dxdy<\infty\},
\end{align*}
where $\rho_{\mathbb{C}^n}(t,z)=\frac{1}{(\pi t)^n}e^{-\frac{|z|^2}{t}}$, and an integral operator $A_{t}:L^2(\mathbb{R}^n,dx)\rightarrow\mathfrak{F}_{n,t}$ given by
\begin{align*}
A_{t}f(z)=\frac{1}{(\pi t)^{\frac{1}{4}}}\int_{\mathbb{R}^n}e^{\frac{-z^2+2\sqrt{2}z\cdot x-x^2}{2t}}f(x)dx,\, \, \, \,  f\in L^2(\mathbb{R}^n,dx)
\end{align*}
for any $t>0$ and any $n\in\mathbb{N}$. Bargmann proved that $A_{t}$ is a unitary operator. The space $\mathfrak{F}_{n,t}$ is called the Segal-Bargmann space and the operator $A_{t}$ is called the Segal-Bargmann transform since I. E. Segal considered almost the same things at the same time (\cite{s1}). The Segal-Bargmann space, the Segal-Bargmann transform and their generalizations are important research objects in mathematical physics, probability theory and representation theory today with some open problems. We remark that the function $\rho_{\mathbb{C}^n}(t,z)$ satisfies the heat equation on $\mathbb{C}^n\cong\mathbb{R}^{2n}$:
\begin{align*}
\frac{\partial \rho_{\mathbb{C}^n}}{\partial t}=\frac{1}{4}\Delta_{\mathbb{R}^{2n}}\rho_{\mathbb{C}^n},
\end{align*}
where $\Delta_{\mathbb{R}^{2n}}$ is the Euclidean Laplacian on $\mathbb{R}^{2n}$. We call the function $\rho_{\mathbb{C}^n}$ the heat kernel on $\mathbb{C}^n$. After Bargmann's and Segal's works, in 1994, B. Hall observed that an operator slightly modified from Segal-Bargmann transform can be treated as a convolution operator (see \cite{bh1}, \cite{bh2}). We introduce the operator following \cite[Section 6.2]{bh2}. Let $\rho_{\mathbb{R}^n}:(0,\infty)\times\mathbb{R}^n\rightarrow\mathbb{R}$ be a function defined by
\begin{align*}
\rho_{\mathbb{R}^n}(t,x)=\frac{1}{(2\pi t)^{\frac{n}{2}}}e^{-\frac{|x|^2}{2t}}.
\end{align*}
We remark that the function $\rho_{\mathbb{R}^n}$ satisfies the heat equation
\begin{align*}
\frac{\partial \rho_{\mathbb{R}^n}}{\partial t}=\frac{1}{2}\Delta_{\mathbb{R}^n}\rho_{\mathbb{R}^n}
\end{align*}
and that $\rho_{\mathbb{R}^n}$ has an analytic continuation to $\mathbb{C}^n$. We call the function $\rho_{\mathbb{R}^n}$ the heat kernel on $\mathbb{R}^n$. We define $\rho_{t}(x)=\rho_{\mathbb{R}^n}(t,x)$. For $t>0$, let $M_{\sqrt{\rho_{t}}}:L^2(\mathbb{R}^n,\rho_{t}dx)\rightarrow L^2(\mathbb{R}^n,dx)$ be a multiplication operator defined by
\begin{align*}
M_{\sqrt{\rho_{t}}}f(x)=\sqrt{\rho_{t}(x)}f(x).
\end{align*}
Then, $M_{\sqrt{\rho_{t}}}$ is a unitary operator. Let $B_{t}=A_{t}\circ M_{\sqrt{\rho_{t}}}$. Then, the operator $B_{t}$ is a unitary operator from $L^2(\mathbb{R}^n,\rho_{t}dx)$ onto $\mathfrak{F}_{n,t}$ given by
\begin{align*}
B_{t}f(z)=\int_{\mathbb{R}^n}\rho_{t}(z-x)f(x)dx.
\end{align*}
This is a convolution operator. We call $B_{t}$ also the Segal-Bargmann transform. Furthermore, Hall introduced a different kind of the Segal-Bargmann transform, where $\mathbb{R}^n$ is replaced by any compact Lie group $K$ and $\mathbb{C}^n$ is replaced by a complex Lie group $K_{\mathbb{C}}$ which is a complexification of $K$ (see \cite[Theorem 1']{bh1}). To understand this theory deeply, we need to calculate examples. To calculate the examples, we should calculate the heat kernels on Lie groups. The universal covering group of $K$ is a direct product group of a vector group $\mathbb{R}^N$ and connected simply-connected compact simple Lie groups. Thus simple examples are the cases where $K$ is a torus or $K$ is a connected simply-connected compact simple Lie group. If $K$ is a torus, we can calculate concretely the heat kernels on $K$ and $K_{\mathbb{C}}$. Next, we consider the case that $K$ is a connected simply-connected compact simple Lie group. The simplest example is the case $K=SU(2)$. We can calculate the heat kernel on $SU(2)$ (see Theorem \ref{cpt_heat}). However, the heat kernel on $K_{\mathbb{C}}=SL(2,\mathbb{C})$ has not been calculated explicitly as far as we know. Thus, if we can calculate the heat kernel on $SL(2,\mathbb{C})$ explicitly, we will give the first example of the Segal-Bargmann transform for a simple Lie group $K$. In this way, we pose the problem of calculating the heat kernel on $SL(2,\mathbb{C})$.

Generalizing the problem, we consider the heat equation and the heat kernel on $SL(2,\mathbb{R})$. The heat kernel on $SL(2,\mathbb{R})$ seems not to be calculated either. We think that this problem is easier than the one for $SL(2,\mathbb{C})$. Indeed, since we have found a good approach to the problem for $SL(2,\mathbb{R})$ in this paper, we may apply it similarly to the problem for $SL(2,\mathbb{C})$ in future and give the example of the Segal-Bargmann transform. In addition, we can understand some open problems about Segal-Bargmann transform deeply.

Before the discussion in details, we recall heat problems on Riemannian manifolds $M$ in two cases. The first case is when $M$ is a compact Lie group. This situation is discussed in Stein's book (see \cite[Chapter 2, Theorem 1]{st}). To know the heat kernel, we calculate the heat semigroup by using Peter-Weyl Theorem. The second case is when $M$ is a Riemannian symmetric space. This situation is discussed by Gangolli (see \cite[Proposition 3.1]{gan}). In this paper, we use the spherical transform on symmetric spaces (see \cite[Chapter 4]{helg2}). Our work is a generalization of both. In fact, we use Peter-Weyl theorem in Section 3 and general spherical transforms in Section 5.

In this paper, we use a general method to calculate the heat kernel on $SL(2,\mathbb{R})$. We expect that the heat kernel on $SL(2,\mathbb{C})$ (more generally, a non-compact semisimple Lie group $G$ having a multiplicity-free subgroup $K$) is calculated by a similar way in future.

Let us explain contents of each section. In Section 2, we introduce Hall's Segal-Bargmann transform. This section is based on \cite{bh1} mainly. In section 3, we discuss the heat kernels on Lie groups. We formulate the heat equation in terms of Riemannian symmetric pairs in Section 3.1.  Next, we show existence and uniqueness of the heat kernel on Lie groups. In Section 4, we give a decomposition of the $L^2$ space on a semisimple Lie group. Each subspace can be understood as a space of sections of a homogeneous vector bundle. In Section 5, we discuss the spherical transforms for homogeneous vector bundles. This theory is introduced by R. Camporesi firstly (see \cite{ca1}). In Section 6, we discuss the Helgason-Fourier transform for homogeneous vector bundles. This theory is also introduced by R. Camporesi (see \cite{ca2}). In Section 7, we calculate the heat kernel on $SL(2,\mathbb{R})$.

\section*{Acknowledgement}
I am very grateful to Professor Hideyuki Ishi for giving me advices in writing this paper. I am also grateful to Professors Hiroshi Oda and Nobukazu Shimeno for giving me insightful suggestions related to this research.

\section{Segal-Bargmann transform}\label{SBtrans}
Let $K$ be a connected compact Lie group, $dk$ be the normalized Haar measure on $K$, and $\mathfrak{k}$ be the Lie algebra of $K$. We take an $\mathrm{Ad}(K)$ invariant inner product $\langle , \rangle_\mathfrak{k}$ on $\mathfrak{k}$. Then $\langle , \rangle_\mathfrak{k}$ induces a $K$ bi-invariant Riemannian metric on $K$. Let $\{ X_1, ... , X_m\}$ be an orthonormal basis of $\mathfrak{k}$, and $\tilde{X}_1, ... , \tilde{X}_m$ be left invariant vector fields on $K$ induced by the elements $X_1, ... , X_m$. Let $\Delta_K = \sum_{j = 1}^{m}\tilde{X}_j\circ\tilde{X}_j$. The differential operator $\Delta_K$ is a Laplace-Beltrami operator on $K$ as a Riemannian manifold (see \cite[Theorem 1]{ura}).
\begin{defi}(cf. \cite[(10)]{bh1}, \cite{ne1}, \cite[Chapter 2]{st})\\
We call a function $\rho_K :(0,\infty)\times K \rightarrow (0,\infty)$ a heat kernel on $K$ if $\rho_K$ satisfies the conditions below:
\begin{eqnarray}
&\mathrm{(i)}& \frac{\partial}{\partial t}(\rho_K * f)=\frac{1}{2}\Delta_K (\rho_K * f), \nonumber \\
&\mathrm{(ii)}& \parallel \rho_K * f - f \parallel _{L^2(K, dk)} \rightarrow 0 \, \, (t \rightarrow +0),\, \, f \in L^2(K,dk). \nonumber
\end{eqnarray}
Here, $f_1 * f_2 (k) = \int_K f_1(k'^{-1}k)f_2(k')dk'.$
\end{defi}
A classical theorem gives a series expression of the heat kernel $\rho_K$.
\begin{thm}(cf. \cite[Chapter 2, Theorem 1]{st})\label{cpt_heat}\\
The heat kernel $\rho_K$ exists uniquely, and it is given by 
\begin{align*}
\rho_K(t,k)=\sum_{\tau \in \hat{K}}(\mathrm{dim}V_{\tau})e^{-c_{\tau}t}\chi_{\tau}(k),
\end{align*}
where $\hat{K}$ is the set of equivalence classes of irreducible unitary representations of $K$ and $c_{\tau}$ is a non-negative real number determined by an equation $\sum_{j=1}^{m}d\tau(X_j)^2=-c_{\tau}id_{V_{\tau}}$.
\end{thm}
We shall introduce a specific complexification of the compact Lie group $K$.
\begin{thm}(cf. \cite[Chapter 17.5, Theorem 5.1]{hoc})\\
A complexification $(K_{\mathbb{C}}, \iota : K\rightarrow K_{\mathbb{C}})$ of $K$ with the following property exists: If $H$ is a complex Lie group and $\phi : K \rightarrow H$ is a homomorphism, there exists a unique holomorphic homomorphism $\tilde{\phi} : K_{\mathbb{C}}\rightarrow H$ such that $\phi=\tilde{\phi}\circ\iota$. Moreover, such a complexification is unique up to a complex Lie group isomorphism.
\end{thm}
We call $K_{\mathbb{C}}$ a universal complexification of $K$.
\begin{thm}(cf. \cite[Proposition 7.5.]{knapp})\\
If $K$ is semisimple, a complexification of $K$ exists and is unique up to a complex Lie group isomorphism. In particular, any complexification is universal.
\end{thm}
The heat kernel $\rho_K$ has an analytic continuation to $K_{\mathbb{C}}$ (see \cite[Proposition 1]{bh1}). Next, we fix notation related to $K_{\mathbb{C}}$. Let $dg$ be a Haar measure on $K_{\mathbb{C}}$ and $\mathfrak{k}_{\mathbb{C}}$ be the Lie algebra of $K_{\mathbb{C}}$. We define an inner product $\langle , \rangle_{\mathfrak{k}_{\mathbb{C}}}$ by
\begin{align*}
\langle X + iX', X'' + iX''' \rangle_{\mathfrak{k_{\mathbb{C}}}} = \langle X, X'' \rangle_{\mathfrak{k}} + \langle X', X''' \rangle_{\mathfrak{k}},\, \, X, X', X'', X'''\in \mathfrak{k}.
\end{align*}
Let $Y_{j}=iX_{j} (1\leq j \leq m)$. Then, $\{ X_1, ... , X_m, Y_1, ... , Y_m\}$ is an orthonormal basis of $\mathfrak{k}_{\mathbb{C}}$.
\begin{rem}
A decomposition $\mathfrak{k}_{\mathbb{C}}=\mathfrak{k}\oplus i\mathfrak{k}$ is exactly the Cartan decomposition for the Riemannian symmetric pair $(K_{\mathbb{C}}, K)$.
\end{rem}
The inner product $\langle , \rangle_{\mathfrak{k}_{\mathbb{C}}}$ induces a left $K_{\mathbb{C}}$ invariant and right $K$ invariant Riemannian metric on $K_{\mathbb{C}}$.\, Let $\Delta_{K_{\mathbb{C}}} = \sum_{j = 1}^{m}(\tilde{X_j}\circ\tilde{X_j}+ \tilde{Y_j}\circ\tilde{Y_j})$. The differential operator $\Delta_{K_{\mathbb{C}}}$ is a Laplace-Beltrami operator on $K_{\mathbb{C}}$ as a Riemannian manifold. Following Hall \cite{bh1}, we give a definition of a heat kernel and recall an existence and uniqueness theorem which is slightly different from the ones for $K$.
\begin{defi}\label{cpxheat}
We call a function $\rho_{K_{\mathbb{C}}} :(0,\infty)\times K_{\mathbb{C}} \rightarrow (0,\infty)$ a heat kernel of $K_{\mathbb{C}}$ if $\rho_{K_{\mathbb{C}}}$ satisfies the conditions below.
\begin{eqnarray}
&\mathrm{(i)}& \frac{\partial}{\partial t}(\rho_{K_{\mathbb{C}}} * f)=\frac{1}{4}\Delta_{K_{\mathbb{C}}} (\rho_{K_{\mathbb{C}}} * f),\nonumber \\
&\mathrm{(ii)}& \parallel \rho_{K_{\mathbb{C}}} * f - f \parallel _{L^2(K_{\mathbb{C}}, dg)} \rightarrow 0 \, \, (t \rightarrow +0),\, \,  f \in L^2(K_{\mathbb{C}},dg). \nonumber
\end{eqnarray}
\end{defi}
\begin{thm}(\cite{ne1})
Let $A$ be a closure of $\frac{1}{4}\Delta_{K_{\mathbb{C}}}$ as a operator on $L^2(K_{\mathbb{C}},dg)$ and $e^{tA}$ be a semigroup generated by $A$. Then, $e^{tA}$ is a convolution operator and its integral kernel is a heat kernel of $K_{\mathbb{C}}$ (we denote it $\rho_{K_{\mathbb{C}}}$).
\end{thm}
Now, we recall B. Hall's theorem. Let $HL^2(K_{\mathbb{C}},\alpha(g)dg)$ be the weighted Bergman space of holomorphic $L^2$ functions on $K_{\mathbb{C}}$ with respect to the weight $\alpha$ which is a positive continuous function on $K_{\mathbb{C}}$.
\begin{thm}(\cite[Theorem 1']{bh1})\\
Fix a positive number $t>0$. Let $B_{t}: L^2(K, \rho_{K}(t,k)dk)\rightarrow HL^2(K_{\mathbb{C}}, \rho_{K_{\mathbb{C}}}(t,g)dg)$ be a linear operator given by
\begin{align*}
B_{t}f(g)=\int_{K}\rho_{K}(t,k^{-1}g)f(k)dk.
\end{align*}
Then, $B_{t}$ is a unitary isomorphism.
\end{thm}
We want to observe closely what happens actually in Theorem 8 for concrete examples. To see the simplest case $K=SU(2)$, we pose a problem.
\begin{prob}\label{SL2Cheat}
Calculate the heat kernel on $SL(2,\mathbb{C})$ explicitly.
\end{prob}

\section{Heat equations}

\subsection{Heat equations on semisimple Lie groups}
In this section, we generalize Problem \ref{SL2Cheat}. Let $G$ be a noncompact connected semisimple Lie group , $dg$ be a Haar measure on $G$, $K\subset G$ be a maximal compact subgroup and $dk$ be the normalized Haar measure on $K$. The pair $(G,K)$ is a Riemannian symmetric pair (see \cite[Chapter 5]{helg1}). Let $\sigma:G\rightarrow G$ be a Cartan involution with respect to the pair $(G,K)$. Let $\mathfrak{g}$ be the Lie algebra of G, $\mathfrak{k}$ be the Lie algebra of $K$ and $d\sigma:\mathfrak{g}\rightarrow\mathfrak{g}$ be a differential of $\sigma$. Then, $\mathfrak{k}=\{X\in\mathfrak{g}\, |\, d\sigma(X)=X\}$. We denote by $\mathfrak{p}$ the set $\{X\in\mathfrak{g}\, |\, d\sigma(X)=-X\}$. Then, we get the decomposition $\mathfrak{g}=\mathfrak{k}\oplus\mathfrak{p}$ as a linear space. We take an $\mathrm{Ad}(K)$ invariant inner product $\langle , \rangle_\mathfrak{g}$ on $\mathfrak{g}$. Then, $\langle , \rangle_\mathfrak{g}$ induces a left $G$ invariant and right $K$ invariant Riemannian metric on $G$. Let $\{ X_1, ... , X_m\}$ be an orthonormal basis of $\mathfrak{k}$, $\{ Y_1, ... , Y_n\}$ be an orthonormal basis of $\mathfrak{p}$, and $\tilde{X}_1, ... , \tilde{X}_m, \tilde{Y}_1, ... , \tilde{Y}_n$ be left invariant vector fields on $G$ which are induced by elements $X_1, ... , X_m, Y_1, ... , Y_n$. Let $\Delta_G = \sum_{j = 1}^{m}\tilde{X}_j\circ\tilde{X}_j +\sum_{k = 1}^{n}\tilde{Y}_k\circ\tilde{Y}_k$. The differential operator $\Delta_G$ is a Laplace-Beltrami operator on $G$ as a Riemannian manifold. 
\begin{defi}\label{riemannheat}
We call a function $\rho_{G} :(0,\infty)\times G \rightarrow (0,\infty)$ a heat kernel on $G$ if $\rho_{G}$ satisfies the conditions below:
\begin{eqnarray}
&\mathrm{(i)}& \frac{\partial}{\partial t}(\rho_{G} * f)=\Delta_{G} (\rho_{G} * f),\nonumber \\
&\mathrm{(ii)}& \parallel \rho_{G} * f - f \parallel _{L^2(G, dg)} \rightarrow 0 \, \, (t \rightarrow +0),\, \,  f \in L^2(G,dg). \nonumber
\end{eqnarray}
\end{defi}
\begin{rem}
Definition \ref{cpxheat} in Section \ref{SBtrans} is a special example of Definition \ref{riemannheat} because $(K_{\mathbb{C}},K)$ is a Riemannian symmetric pair.
\end{rem}
Now, we pose a generalized problem.
\begin{prob}\label{G_heat}
Calculate the heat kernel on $G$ explicitly. Especially, calculate the heat kernel on $SL(2,\mathbb{R})$.
\end{prob}

\subsection{Existence and uniqueness of the heat kernel}
First, we recall basic notions. Let $M$ be a Riemannian manifold. Let $\mathfrak{X}(M)$ be the set of all vector fields on $M$, $\nabla:\mathfrak{X}(M)\times\mathfrak{X}(M)\rightarrow\mathfrak{X}(M)$ be the Levi-Civita connection, $R(X,Y):=\nabla_{X}\nabla_{Y}-\nabla_{Y}\nabla_{X}-\nabla_{[X,Y]}$ be the curvature of the pair $(X,Y)\in\mathfrak{X}(M)\times\mathfrak{X}(M)$ and let $\mathrm{Ric}:\mathfrak{X}(M)\times\mathfrak{X}(M)\rightarrow C^{\infty}(M)$ be the Ricci curvature given by
\begin{align*}
\mathrm{Ric}(X,Y)=\sum_{j=1}^{m}\langle R(X,E_{j})Y,E_{j} \rangle,\, \,  X, Y\in\mathfrak{X}(M),
\end{align*}
where $\{E_j\}_{j=1}^m$ is a frame field on $M$ (see \cite[Lemma 52 in Chapter 3]{neil}). We recall an important notion about the Ricci curvature.
\begin{defi}\label{riccibound}(cf. \cite[Section 3.3]{cha})
We say that the Ricci curvature is bounded from below if there exists $\kappa\in\mathbb{R}$ such that
\begin{align*}
\mathrm{Ric}(u,u)\geq\kappa||u||_{x}^2,\, \, u\in T_{x}M, x\in M.
\end{align*}
\end{defi}

Next, we discuss the heat kernel. Let $dv$ be the volume form on $M$ and $\Delta$ be the Laplace-Beltrami operator.
\begin{defi}
A function $\rho$ on $(0,\infty)\times M\times M$ is called a heat kernel if the following holds: for each $f_{0}\in L^2(M,dv)$, if we put $\displaystyle f(t,x):=\int_{M}\rho(t,x,y)f_{0}(y)dv(y)\, (t>0)$, then $f(t,\bullet)\in L^2(M,dv)$ for all $t>0$ and $f$ satisfies the conditions below.
\begin{eqnarray}
&\mathrm{(i)}& \frac{\partial f}{\partial t}=\Delta f,\nonumber \\
&\mathrm{(ii)}& ||f(t,\bullet)-f_{0}||_{L^2(M,dv)}\rightarrow 0\,\, (t\rightarrow +0).\nonumber
\end{eqnarray}
\end{defi}
We show existence and uniqueness of the heat kernel on $G$. Dodziuk's work \cite{dod} tells us that if $M$ is complete with Ricci curvature bounded from below, then there exists a unique heat kernel. Thus, it is enough to show that $G$ is a complete Riemannian manifold with Ricci curvature bounded from below. The statements below are well-known, while we write their proofs for completeness.
\begin{lem}
$G$ is complete.
\end{lem}
\begin{proof}
We show that every Cauchy sequence converges. Let $d$ be the Riemannian distance on $G$ and $\{g_{\nu}\}_{\nu=1}^{\infty}$ be any Cauchy sequence in $G$. Since $G$ is locally compact, there exists $r>0$ such that the open subset $U=\{g\in G|d(g,e)<r\}$ is relatively compact. Since $\{g_{\nu}\}_{\nu=1}^{\infty}$ is a Cauchy sequence, there exists $N\in\mathbb{N}$ such that $d(g_{\nu},g_{\nu'})<r$ for all $\nu, \nu'\geq N$. For $g\in G$, let $l_{g}:G\ni x\mapsto gx\in G$. These are isometries. We consider the sequence $\{l_{g_{N}^{-1}}(g_{\nu})\}_{\nu=1}^{\infty}$. This is also a Cauchy sequence and $\{l_{g_{N}^{-1}}(g_{\nu})\}_{\nu=N}^{\infty}\subset U$. Since $U$ is relatively compact, the sequence $\{l_{g_{N}^{-1}}(g_{\nu})\}_{\nu=1}^{\infty}$ converges in $G$. Then, the sequence $\{g_{\nu}\}_{\nu=1}^{\infty}$ converges to $\displaystyle l_{g_{N}}(\lim_{\nu\to\infty}l_{g_{N}^{-1}}(g_{\nu}))$.
\end{proof}
\begin{lem}
The Ricci curvature is bounded from below.
\end{lem}
\begin{proof}
It is known that $\nabla_{\tilde{X}}\tilde{Y}$ is left invariant for $X, Y\in\mathfrak{g}$ (see \cite[Section 1.3 in Chapter 2]{helg1}). Thus, $R(\tilde{X},\tilde{Y})\tilde{Z}$ is also left invariant for $X,Y,Z\in\mathfrak{g}$. We put $\{E_{j}\}=\{\tilde{X_{j}}\}_{j=1}^{m}\cup\{\tilde{Y_{k}}\}_{k=1}^{n}$ as a frame field. The value $\mathrm{Ric}(\tilde{X},\tilde{Y})$ is a constant on $G$ for each  $X,Y\in\mathfrak{g}$. Then, value of $\mathrm{Ric}$ and $\kappa$ in Definition \ref{riccibound} is determined on $T_{e}G$.
\end{proof}
\begin{cor}
The heat kernel on $G$ exists uniquely.
\end{cor}

\section{Decomposition of the $L^2$ space}\label{decom}
We introduce a decomposition of $L^2(G,dg)$ to be used later. Since the pair $(G,K)$ is a Riemannian symmetric pair, the map $\mathfrak{p}\times K\ni (Y,k)\mapsto e^{Y}k\in G$ is a diffeomorphism (see \cite[Theorem 6.31]{knapp}). In this situation, $\mathfrak{p}$ is diffeomorphic to $G/K$ naturally. Thus $G/K\times K$ and $G$ are diffeomorphic. Let $\iota:G/K\times K\rightarrow G$ be the diffeomorphism and $dp$ be a left $G$ invariant measure on $G/K$ such that $\iota^{*}dg = dp dk$. Then,
\begin{align*}
L^2(G,dg) \cong L^2(G/K,dp)\otimes L^2(K,dk).
\end{align*}
In this equality, the right hand side means completion of a tensor product space. Using the Peter-Weyl theorem:
\begin{align*}
L^2(K,dk)\cong\sum_{\tau\in \hat{K}}V_{\tau}\otimes V_{\tau}^*,
\end{align*} 
we get
\begin{align*}
L^2(G,dg) \cong \sum_{\tau\in \hat{K}}(L^2(G/K,dp)\otimes V_{\tau})\otimes V_{\tau}^{*}.
\end{align*}
Let
\begin{align*}
L^2(G,\tau)=\{f:G\rightarrow V_{\tau}\, |\, f(gk)=\tau(k^{-1})f(g)\, (k\in K, g\in G), \, \int_{G}||f(g)||^2_{V_{\tau}}dg<\infty \},
\end{align*}
which is regarded naturally as the space of $L^2$-sections of a homogeneous vector bundle $G\times_{K}V_{\tau}$ over $G/K$ (see \cite[Definition 3 in Chapter 4]{mits}). Note that $L^2(G,\tau)$ is preserved by the left translation $L_h\, (h\in G)$ given by $L_hf(g)=f(h^{-1}g)$.
\begin{lem}
The map
\begin{align*}
L^2(G/K,dp)\otimes V_{\tau}\ni f\otimes v \mapsto (\iota(p,k)\mapsto \tau(k^{-1})f(p)v)\in L^2(G,\tau)\, \, \, \, \, \, \, \, \, (1)
\end{align*}
gives an isomorphism as Hilbert spaces:
\begin{align*}
L^2(G/K,dp)\otimes V_{\tau}\cong L^2(G,\tau).
\end{align*}
\end{lem}
\begin{proof}
Let $\{v_{1}, ... , v_{d}\}$ be an orthonormal basis of $V_{\tau}$. For a function $f\in L^2(G,\tau)$ and $g\in G$, we obtain the expression $f(g)=\sum_{s=1}^{d}f_{s}(g)v_{s}$. Then the map
\begin{align*}
L^2(G,\tau)\ni f\mapsto \sum_{s=1}^{d}f_{s}(\iota(\bullet, e_{K}))\otimes v_{s}\in L^2(G/K,dp)\otimes V_{\tau}
\end{align*}
is an inverse of the map $(1)$.
\end{proof}
Therefore we have
\begin{align*}
L^2(G,dg)\cong\sum_{\tau\in \hat{K}}L^2(G,\tau)\otimes V_{\tau}^*.
\end{align*}
We define an action of $G\times K$ on $L^2(G,dg)$ by
\begin{align*}
(h,k)\cdot f(g)=f(h^{-1}gk),\, \, h\in G,\, k\in K
\end{align*}
and an action of $G\times K$ on $L^2(G,\tau)\otimes V_{\tau}^{*}$ by
\begin{align*}
(h,k)\cdot f\otimes v^{*} =L_hf\otimes \tau^{*}(k)v^{*},\, \, h\in G,\, k\in K.
\end{align*}
In this situation, the isomorphism in the decomposition of $L^2(G,dg)$ above is an intertwining operator as $G\times K$ modules. We treat the differential operator $\Delta_{G}$ as an operator on $L^2(G,dg).$ Then, $\Delta_{G}$ preserves each subspace which is isomorphic to $L^2(G,\tau)\otimes V_{\tau}^{*}$. Thus, we can reduce the heat equation on $G$ to heat equations on homogeneous vector bundles $G\times_{K}V_{\tau}$.

\section{The spherical transforms}\label{ST}

\subsection{A multiplicity free subgroup}
Let us recall Camporesi's work here. Let $G$ be a locally compact group satisfying the 2nd axiom of countability and $K$ be a compact subgroup of $G$. For $\tau\in\hat{K}$ and $U\in\hat{G}$, let $m(\tau,U)$ be the multiplicity of $\tau$ in $U|_{K}$.
\begin{defi}(\cite{koo})
A compact subgroup $K$ of $G$ is said to be a multiplicity free subgroup if $m(\tau,U)\leq 1$ for all $\tau\in\hat{K}$ and all $U\in\hat{G}$.
\end{defi}
Examples of pairs $(G,K)$ where $K$ is a multiplicity free subgroup of $G$ are found in \cite[Theorem 1]{koo}. One of the examples is $(SU(1,1), S(U(1)\times U(1)))$. Let $c=\frac{1}{\sqrt{2}}
\begin{pmatrix}
1 & i \\
i & 1
\end{pmatrix}
$. Since $cSU(1,1)c^{-1}=SL(2,\mathbb{R})$ and $cS(U(1)\times U(1))c^{-1}=SO(2)$, the compact group $SO(2)$ is a multiplicity free subgroup of $SL(2,\mathbb{R})$.

\subsection{The spherical transforms}
We discuss the spherical transform of ``radial systems of sections'' (denoted by $C_{0}^{\infty}(G,\tau,\tau)$ below, see \cite{ca1}). Let $G$ be a connected noncompact semisimple Lie group with finite center and $K$ be a maximal compact subgroup of $G$. We assume that $K$ is a multiplicity free subgroup of $G$. Let $\tau\in\hat{K}$ and
\begin{flushleft}
$C_{0}^{\infty}(G,\tau,\tau)=$\\
$\{F\in C_{0}^{\infty}(G,\mathrm{End}(V_{\tau}))\, |\, F(k_{1}gk_{2})=\tau(k_{2}^{-1})F(g)\tau(k_{1}^{-1}),\, \, x \in G,\, k_{1},\, k_{2}\in K\}$.\\
\end{flushleft}
For any $F\in C_{0}^{\infty}(G,\tau,\tau)$ and $v\in V_{\tau}$, a vector-valued function $f(g)=F(g)v\, (g\in G)$ defines a section of a homogeneous vector bundle $G\times_{K}V_{\tau}$. Since $G=KAK$ (Cartan decomposition), the value of $f$ is determined on $A$. Because of this observation, $f$ is called  a ``radial section'' and $F$ is called a ``radial system''. The convolution product on $C_{0}^{\infty}(G,\tau,\tau)$ is defined by
\begin{align*}
F_{1}*F_{2}(g)=\int_{G}F_{1}(g'^{-1}g)F_{2}(g')dg'.
\end{align*}
In fact, $F_{1}*F_{2}\in C_{0}^{\infty}(G,\tau,\tau)$ because
\begin{eqnarray}
F_{1}*F_{2}(k_{1}gk_{2})&=&\int_{G}F_{1}(g'^{-1}(k_{1}gk_{2}))F_{2}(g')dg' \nonumber \\
&=&\int_{G}F_{1}(g'^{-1}gk_{2})F_{2}(k_{1}g')dg' \nonumber \\
&=&\tau(k_{2}^{-1})\Bigl( \int_{G}F_{1}(g'^{-1}g)F_{2}(g')dg'\Bigr)\tau(k_{1}^{-1}) \nonumber \\
&=&\tau(k_{2}^{-1})F_{1}*F_{2}(g)\tau(k_{1}^{-1}). \nonumber
\end{eqnarray}
Let $\hat{G}(\tau)=\{ U\in\hat{G}\, |\, m(\tau,U)\neq 0\}$. We fix $U\in\hat{G}(\tau)$ and we define the spherical function on $G$ (see \cite[(10)]{ca1}). Let $H_{U}$ be a representation space of $U$, $H_{\tau}\subset H_{U}$ be
the isotypic component of $\tau$ and $P_{\tau}:H_{U}\rightarrow H_{\tau}$ be the orthogonal projection operator. Since $m(\tau,U)=1$, there exists an isomorphism $\iota_{\tau}:V_{\tau}\rightarrow H_{\tau}$. Let $\Phi_{\tau}^{U}:G\rightarrow \mathrm{End}(V_{\tau})$ be
\begin{align*}
\Phi_{\tau}^{U}(g)=\iota_{\tau}^{-1}P_{\tau}U(g)\iota_{\tau}.
\end{align*}
We call this function $\Phi_{\tau}^{U}$ a spherical function of type $\tau$ (\cite{ca1}). Next, we define the spherical transform.
\begin{defi}(\cite[Definition 3.3]{ca1})
The spherical transform of $F\in C_{0}^{\infty}(G,\tau,\tau)$ is the function $\hat{F}:\hat{G}(\tau)\rightarrow \mathbb{C}$ defined by
\begin{align*}
\hat{F}(U)=\frac{1}{\mathrm{dim}V_{\tau}}\int_{G}\mathrm{tr}(\Phi_{\tau}^{U}(g)F(g))dg.
\end{align*}
\end{defi}
The spherical transform preserves products.
\begin{thm}(\cite[After Lemma 3.4]{ca1})\label{product}
For any $F_{1}, F_{2}\in C_{0}^{\infty}(G,\tau,\tau)$, we have
\begin{align*}
\widehat{(F_{1}*F_{2})}(U)=\hat{F_{1}}(U)\hat{F_{2}}(U),\, \,U\in\hat{G}(\tau).
\end{align*}
\end{thm}
The spherical transform has an inverse map. To describe this, we recall the Plancherel measure on $\hat{G}$. For any $f\in C_{0}(G)$ and $U\in\hat{G}$, let $U(f)$ be the linear operator on $H_{U}$ defined by
\begin{align*}
U(f)=\int_{G}f(g)U(g)dg
\end{align*}
(see \cite[Section 4.1]{wa1}).
\begin{thm}(cf. \cite[Theorem 7.2.1.1]{wa2})
There exists the measure $d\mu$ on $\hat{G}$ uniquely such that
\begin{align*}
\int_{G}|f(g)|^2dg=\int_{\hat{G}}\mathrm{tr}(U(f)U(f)^{*})d\mu(U),\, \, f\in L^1(G,dg)\cap L^2(G,dg).
\end{align*}
\end{thm}
Such $d\mu$ is called the Plancherel measure on $\hat{G}$.
\begin{thm}(\cite[Theorem3.9]{ca1})
For any $F\in C_{0}^{\infty}(G,\tau,\tau)$, we have
\begin{align*}
F(g)=\frac{1}{\mathrm{dim}V_{\tau}}\int_{\hat{G}(\tau)}\Phi_{\tau}^{U}(g^{-1})\hat{F}(U)d\mu(U).
\end{align*}
\end{thm}
In addition, the spherical transform is a unitary map.
\begin{thm}(\cite[Corollary 3.10]{ca1})\label{plancherel}
The spherical transform extends to a unitary map $L^2(G,\tau,\tau)\rightarrow L^2(\hat{G}(\tau),d\mu)$.
\end{thm}

\section{The Helgason-Fourier transforms}

\subsection{Principle series representations}\label{Principle series}
First, we discuss parabolic subgroups according to \cite[Section 7.7]{knapp}. Let $G$ be a linear connected reductive group, $K$ be the compact subgroup, $\Theta$ be the global Cartan involution, $B$ be a nondegenerate bilinear form on $\mathfrak{g}$ and $\mathfrak{g}=\mathfrak{k}\oplus\mathfrak{p}$ be the Caran decomposition. Let $\mathfrak{a}\subset\mathfrak{p}$ be a maximal abelian subspace, $\Sigma(\mathfrak{g},\mathfrak{a})$ be the set of restricted roots, $\Sigma^{+}(\mathfrak{g},\mathfrak{a})$ be a set of positive restricted roots with some linear order, $\mathfrak{m}\subset\mathfrak{k}$ be the centralizer of $\mathfrak{a}$ in $\mathfrak{k}$, and $\mathfrak{n}=\oplus_{\lambda\in\Sigma^{+}(\mathfrak{g},\mathfrak{a})}\mathfrak{g}_{\lambda}$. Let  $A$ and $N$ be the analytic subgroups of $G$ corresponding to $\mathfrak{a}$ and $\mathfrak{n}$ respectively, and $M\subset G$ be the centralizer of $A$. Then, $\mathfrak{m}=\mathrm{Lie}(M)$ and we get an Iwasawa decomposition $G=KAN$ (see \cite[Proposition 7.31]{knapp}) and a closed subgroup $Q=MAN$ called a minimal parabolic subgroup (see \cite[After Proposition 7.31]{knapp}). Let $\mathfrak{q}=\mathfrak{m}\oplus\mathfrak{a}\oplus\mathfrak{n}$. A Lie subalgebra $\mathfrak{q}'\subset\mathfrak{g}$ is called a $\mathfrak{q}$-parabolic subalgebra if $\mathfrak{q}\subset\mathfrak{q}'$. We call it just a parabolic subalgebra for simplicity in what follows. Let $\Pi(\mathfrak{g},\mathfrak{a})$ be the set of simple roots.
\begin{prop}(cf. \cite[Proposition 7.76]{knapp})
Let $\mathfrak{q}'$ be a parabolic subalgebra. Then, there exists a subset $\Pi'\subset\Pi(\mathfrak{g},\mathfrak{a})$ such that
\begin{align*}
\mathfrak{q}'=\mathfrak{m}\oplus\mathfrak{a}\oplus(\bigoplus_{\beta\in\Gamma}\mathfrak{g}_{\beta}),
\end{align*}
where $\Gamma=\Sigma^{+}(\mathfrak{g},\mathfrak{a})\cup(\Sigma(\mathfrak{g},\mathfrak{a})\cap\mathrm{span}(\Pi'))$.
\end{prop}
We put
\begin{eqnarray}
\mathfrak{a}'&=&\bigcap_{\beta\in\Gamma\cap(-\Gamma)}\mathrm{Ker}\beta, \nonumber \\
\mathfrak{m}'&=&\mathfrak{m}\oplus\mathfrak{a}'^{\perp}\oplus(\bigoplus_{\beta\in\Gamma\cap(-\Gamma)}\mathfrak{g}_{\beta}), \nonumber \\
\mathfrak{n}'&=&\bigoplus_{\beta\in\Gamma\setminus(-\Gamma)}\mathfrak{g}_{\beta}.\nonumber
\end{eqnarray}
Let $A'$ and $N'$ be analytic subgroups such that $\mathrm{Lie}(A')=\mathfrak{a'}$ and $\mathrm{Lie}(N')=\mathfrak{n'}$. Let $\mathfrak{z}(\mathfrak{a}')\subset \mathfrak{g}$ be the centralizer of $\mathfrak{a}'$, $Z(A')\subset G$ be the centralizer of $A'$, $M'_{ss}\subset G$ be the analytic subgroup corresponding to $[\mathfrak{z}(\mathfrak{a}'),\mathfrak{z}(\mathfrak{a}')]$ and $M'$ be the group $(K\cap Z(A'))M'_{ss}$. Then we have $\mathfrak{q'}=\mathfrak{m'}\oplus\mathfrak{a'}\oplus\mathfrak{n'}$ (see \cite[After Proposition 7.76]{knapp}) and $Q'=M'A'N'$ is a closed subgroup of $G$ (see \cite[After Corollary 7.81]{knapp}). These decompositions are called the Langlands decompositions. Moreover, we call $Q'$ a parabolic subgroup. Let $\theta=d\Theta$. A parabolic subgroup $Q'=M'A'N'$ is called cuspidal if there is a $\theta$ stable compact Cartan subalgebra in $\mathfrak{m}'$.

Next, we discuss induced representations according to \cite[Section 7.1]{knapp2}. We put $(\mathfrak{a}')^{*}=\mathrm{Hom}_{\mathbb{R}}(\mathfrak{a}',\mathbb{R})$ and $(\mathfrak{a}'_{\mathbb{C}})^{*}=\mathrm{Hom}_{\mathbb{R}}(\mathfrak{a}',\mathbb{C})$. We take $\sigma'\in\hat{M'}$, $\nu'\in(\mathfrak{a}'_{\mathbb{C}})^{*}$ and $\rho_{\mathfrak{a}'}=\frac{1}{2}\sum_{\alpha\in\Sigma^{+}(\mathfrak{g},\mathfrak{a}')}(\mathrm{dim}\mathfrak{g}_{\alpha})\alpha$. Let $V_{\sigma'}$ be a representation space of $\sigma'$, $\mathscr{H}(Q',\sigma',\nu')$ be
the space of $V_{\sigma'}$-valued measurable functions $F$ such that $F(gm'a'n')=e^{-(\nu'+\rho_{\mathfrak{a}'})(\mathrm{log}a)}\sigma'(m)^{-1}F(g)$ for a.e. $g\in G$, $m'\in M'$, $a'\in A'$ and $n'\in N'$ with $||F||^2:=\int_{K}|F(k)|^2dk<\infty$. The induced representation $U(Q',\sigma',\nu')=\mathrm{ind}_{Q'}^{G}(\sigma'\otimes e^{\nu'}\otimes 1)$ is defined on $\mathscr{H}(Q',\sigma',\nu')$ by 
\begin{align*}
U(Q',\sigma',\nu')(g)F(h)=F(g^{-1}h),\, \, g, h\in G.
\end{align*}
Then, the pair $(U(Q',\sigma',\nu'),\mathscr{H}(Q',\sigma',\nu'))$ is a continuous representation of $G$ (see \cite[Section 7.2]{knapp2}), which is called principle series. It is known that $U(Q',\sigma',\nu')$ is unitary if $\nu'$ is pure imaginary valued on $\mathfrak{a}'$. Moreover, if $\sigma'$ has a real infinitesimal character and no root in $\Sigma(\mathfrak{g},\mathfrak{a}')$ is orthogonal to $\nu'$, then $U(Q',\sigma',\nu')$ is irreducible (see \cite[Theorem 14.93]{knapp2}).

\subsection{The Helgason-Fourier transforms}\label{helgfourier}
First, we discuss the Helgason-Fourier transforms. We assume that $G$ is a semisimple Lie group having a multiplicity-free subgroup $K$ and that $Q'$ is cuspidal. We take $A_{1}, N_{1}\subset G$ so that $A=A'A_{1}$ and $N=N'N_{1}$ as in \cite[Proposition 7.14]{knapp2}. It is known that $MA_{1}N_{1}$ is a minimal parabolic subgroup of $M'$ (see \cite[Section 8.10]{knapp2}).
\begin{prop}(\cite[Proposition 4.1]{ca2})
Let $\sigma'\in\hat{M'}$ be a discrete series representation (see \cite[Theorem 8.51]{knapp2}). Then, there exists $\tilde{\sigma}'\in\hat{M}$ and $\mu_{1}\in\mathfrak{a}_{1}^{*}$ such that $\sigma'$ is infinitesimally equivalent with a subrepresentation of $\mathrm{ind}_{MA_{1}N_{1}}^{M'}(\tilde{\sigma'}\otimes e^{\mu_{1}}\otimes 1)$.
\end{prop}
We write $g=\bold{k}(g)e^{H(g)}n(g)\in KAN$ for each $g\in G$. Let $\tau\in\hat{K}$,
\begin{align*}
C_{0}^{\infty}(G,\tau)=\{f:G\rightarrow V_{\tau}\, |\, f(gk)=\tau(k^{-1})f(g)\, (k\in K, g\in G), \, f\in C_{0}^{\infty}(G,V_{\tau}) \}
\end{align*}
and $F^{\lambda}:G\rightarrow\mathrm{End}_{\mathbb{C}}(V_{\tau})$ be a map given by $F^{\lambda}(g)=e^{\lambda(H(g))}\tau(\bold{k}(g))$ for each $g\in G$ and $\lambda\in\mathfrak{a}_{\mathbb{C}}^{*}$. We define the inner product on $\mathrm{End}_{M}(V_{\tau},V_{\tilde{\sigma}'})$ by $\langle S,T\rangle=\frac{1}{\mathrm{dim}V_{\tilde{\sigma}'}}\mathrm{tr}(S^{*}T)$, where $S^{*}$ is an adjoint operator of $S$. We take an orthonormal basis $\{T_{\xi}\}$ of $\mathrm{End}_{M}(V_{\tau},V_{\tilde{\sigma}'})$ and $T_{\tilde{\sigma}'}=\sum_{\xi}T_{\xi}^{*}T_{\xi}$.
\begin{defi}(\cite[(3.18)]{ca2})
For $f\in C_{0}^{\infty}(G,\tau)$, we define the function $\tilde{f}:\mathfrak{a}^{*}_{\mathbb{C}}\times K\to V_{\tau}$ by
\begin{align*}
\tilde{f}(\lambda,k)=\int_{G}F^{i\bar{\lambda}-\rho_{\mathfrak{a}}}(g^{-1}k)^{*}f(g)dg.
\end{align*}
We say $\tilde{f}$ is the Helgason-Fourier transform of $f$.
\end{defi}
\begin{thm}\label{Plancherel_formula}(\cite[After Theorem 4.3]{ca2})
Let $f_{1}, f_{2}\in C_{0}^{\infty}(G,\tau)$. Then, there exists $c_{Q'}>0$ for each $Q'$ such that
\begin{flushleft}
$\langle f_{1},f_{2}\rangle_{L^2(G,dg)}$\\
$\displaystyle =\sum_{Q'}c_{Q'}\sum_{\sigma'}\frac{1}{\mathrm{dim}V_{\tilde{\sigma}'}}\int_{\mathfrak{a}'^{*}\times K}\langle T_{\tilde{\sigma}'}\tilde{f}_{1}(\nu'+i\mu_{1},k),T_{\tilde{\sigma}'}\tilde{f}_{2}(\nu'-i\mu_{1},k)\rangle_{V_{\tau}}p_{\sigma'}(\nu')d\nu'dk$,
\end{flushleft}
where $p_{\sigma'}(\nu')d\nu'=d\mu(U(Q',\sigma',i\nu'))$ and each $\sigma'$ is a discrete series of $M'$ such that $U(Q',\sigma',i\nu')|_{K}$ contains $\tau$.
\end{thm}

Next, we describe relations between the Helgason-Fourier transforms and the spherical transforms in a special case. Let $\Psi\in C_{0}^{\infty}(G,\tau,\tau)$, $v\in V_{\tau}$ and $\psi(g)=\Psi(g)v$ for $g\in G$. Then, $\psi\in C_{0}^{\infty}(G,\tau)$. In this situation, we get
\begin{align*}
\tilde{\psi}(\nu,k)=\sum_{\sigma\in\hat{M}, \sigma\subset\tau|_{M}}\ ^{t}\hat{\Psi}(U(Q,\sigma,i\nu))P_{\sigma}\tau(k^{-1})v,\, \, \nu\in\mathfrak{a}^{*}, k\in K
\end{align*}
(\cite[(5.28)]{ca2}). In relation to this, we show an important formula which does not appear in \cite{ca2}. This is generalization of \cite[Lemma 1.4 in Chapter 3]{helg3}. Let $\Psi\in C_{0}^{\infty}(G,\tau,\tau)$ and $f\in C_{0}^{\infty}(G,\tau)$. We define the convolution between $\Psi$ and $f$ by
\begin{align*}
\Psi*f(g)=\int_{G}\Psi(g'^{-1}g)f(g')dg'.
\end{align*}
\begin{thm}\label{conv_mult}
Let $\Psi\in C_{0}^{\infty}(G,\tau,\tau)$ and $f\in C_{0}^{\infty}(G,\tau)$. Then,
\begin{align*}
\widetilde{\Psi*f}(\nu,k)=\Bigl( \sum_{\sigma\in\hat{M}, \sigma\subset\tau|_{M}}\ ^{t}\hat{\Psi}(U(Q,\sigma,i\nu))P_{\sigma}\Bigr) \tilde{f}(\nu,k),\, \, \nu\in\mathfrak{a}^{*}, k\in K.
\end{align*}
\end{thm}
\begin{proof}
Let $h, h'\in G$ and $l\in K$. We rewrite $hh'l$:
\begin{eqnarray}
hh'l&=&h\bold{k}(h'l)e^{H(h'l)}n(h'l) \nonumber \\
&=&\bold{k}(h\bold{k}(h'l))e^{H(h\bold{k}(h'l))}n(h\bold{k}(h'l))e^{H(h'l)}n(h'l) \nonumber \\
&=&\bold{k}(h\bold{k}(h'l))e^{H(h\bold{k}(h'l))+H(h'l)}e^{-H(h'l)}n(h\bold{k}(h'l))e^{H(h'l)}n(h'l). \nonumber
\end{eqnarray}
Since $A$ normalizes $N$, we get $\bold{k}(hh'l)=\bold{k}(h\bold{k}(h'l))$ and $H(hh'l)=H(h\bold{k}(h'l))+H(h'l)$.
 Then,
\begin{eqnarray}
&&\widetilde{\Psi*f}(\nu,k)\nonumber \\
&=&\int_{G}e^{-(i\nu+\rho_{\mathfrak{a}})(H(g^{-1}k))}\tau(\bold{k}(g^{-1}k))^{-1}\Psi*f(g)dg \nonumber \\
&=&\int_{G}\int_{G}e^{-(i\nu+\rho_{\mathfrak{a}})(H(g^{-1}k))}\tau(\bold{k}(g^{-1}k))^{-1}\Psi(g'^{-1}g)f(g')dgdg'. \nonumber
\end{eqnarray}
We change the variable from $g$ to $g'g$:
\begin{eqnarray}
&&\widetilde{\Psi*f}(\nu,k)\nonumber \\
&=&\int_{G}\int_{G}e^{-(i\nu+\rho_{\mathfrak{a}})(H(g^{-1}g'^{-1}k))}\tau(\bold{k}(g^{-1}g'^{-1}k))^{-1}\Psi(g)f(g')dgdg'. \nonumber
\end{eqnarray}
Next, we use the formula above with $h=g^{-1}$, $h'=g'^{-1}$ and $l=k$:
\begin{eqnarray}
&&\widetilde{\Psi*f}(\nu,k)\nonumber \\
&=&\int_{G}\int_{G}e^{-(i\nu+\rho_{\mathfrak{a}})(H(g^{-1}\bold{k}(g'^{-1}k))+H(g'^{-1}k))}\tau(\bold{k}(g^{-1}\bold{k}(g'^{-1}k)))^{-1}\Psi(g)f(g')dgdg'. \nonumber
\end{eqnarray}
We change the variable from $g$ to $\bold{k}(g'^{-1}k)g$:
\begin{eqnarray}
&&\widetilde{\Psi*f}(\nu,k)\nonumber \\
&=&\int_{G}\int_{G}e^{-(i\nu+\rho_{\mathfrak{a}})(H(g^{-1})+H(g'^{-1}k))}\tau(\bold{k}(g^{-1}))^{-1}\Psi(\bold{k}(g'^{-1}k)g)f(g')dgdg' \nonumber \\
&=&\int_{G}\int_{G}e^{-(i\nu+\rho_{\mathfrak{a}})(H(g^{-1})+H(g'^{-1}k))}\tau(\bold{k}(g^{-1}))^{-1}\Psi(g)\tau(\bold{k}(g'^{-1}k))^{-1}f(g')dgdg'. \nonumber \\
&=&\int_{G}e^{-(i\nu+\rho_{\mathfrak{a}})(H(g^{-1}))}\tau(\bold{k}(g^{-1}))^{-1}\Psi(g)dg\tilde{f}(\nu,k). \nonumber
\end{eqnarray}
Thus, we get
\begin{align*}
\widetilde{\Psi*f}(\nu,k)=\Bigl( \sum_{\sigma\in\hat{M}, \sigma\subset\tau|_{M}}\ ^{t}\hat{\Psi}(U(Q,\sigma,i\nu))P_{\sigma}\Bigr) \tilde{f}(\nu,k).
\end{align*}
\end{proof}

In general, let $Q'$ be any parabolic subgroup, $\Psi\in C_{0}^{\infty}(G,\tau,\tau)$, $v\in V_{\tau}$ and $\psi(g)=\Psi(g)v$ for $g\in G$. It is known that
\begin{align*}
T_{\tilde{\sigma}'}\tilde{\psi}(\nu'-i\mu_{1})=\ ^{t}\hat{\Psi}(U(Q',\sigma',\nu'))T_{\tilde{\sigma}'}\tau(k^{-1})v
\end{align*}
(see \cite[After (5.30)]{ca2}).

\section{The heat kernel on $SL(2,\mathbb{R})$}

\subsection{The heat equation on $SL(2,\mathbb{R})$}
First, we fix some notations and define a differential operator. Let $G=SL(2,\mathbb{R})$ and $K=SO(2)$. Since the map $\theta :G\ni g\mapsto (^{t}g)^{-1}\in G$ is a Cartan involution, the pair $(G,K)$ is a Riemannian symmetric pair. Let $\mathfrak{g}=\mathfrak{sl}(2,\mathbb{R})$, $\mathfrak{k}=\mathfrak{so}(2)$ and $\mathfrak{p}=\{X\in\mathfrak{g}\, |\, d\theta(X)=-X\}$. Then $\mathfrak{g}=\mathfrak{k}\oplus\mathfrak{p}$. Let $\langle, \rangle_{\mathfrak{g}}$ be an inner product on $\mathfrak{g}$ defined by $\langle X, Y \rangle_{\mathfrak{g}}=4\mathrm{tr}(^tXY)$. Then the inner product satisfies the conditions below:
\begin{flushleft}
$\langle X,Y \rangle_{\mathfrak{g}}=-4\mathrm{tr}(XY),\, \, X, Y\in\mathfrak{k}$,\\
$\langle X,Y \rangle_{\mathfrak{g}}=+4\mathrm{tr}(XY),\, \, X, Y\in\mathfrak{p}$,\\
$\langle X,Y \rangle_{\mathfrak{g}}=0,\, \, X\in\mathfrak{k},Y\in\mathfrak{p}$.
\end{flushleft}
These properties are related to the Killing form. Moreover $\langle, \rangle_{\mathfrak{g}}$ is $\mathrm{Ad}(K)$ invariant. Let
\begin{align*}
X_{1}=\frac{1}{\sqrt{8}}
\begin{pmatrix}
0 & -1 \\
1 & 0
\end{pmatrix}
,\, Y_{1}=\frac{1}{\sqrt{8}}
\begin{pmatrix}
1 & 0 \\
0 & -1
\end{pmatrix}
,\, Y_{2}=\frac{1}{\sqrt{8}}
\begin{pmatrix}
0 & 1 \\
1 & 0
\end{pmatrix}.
\end{align*}
Then, $\{X_{1}\}$ is an orthonormal basis of $\mathfrak{k}$ and $\{Y_{1},Y_{2}\}$ is an orthonormal basis of $\mathfrak{p}$. Thus the differential operator which we should consider is
\begin{align*}
\Delta_{G}=\tilde{X_{1}}\circ\tilde{X_{1}}+\tilde{Y_{1}}\circ\tilde{Y_{1}}+\tilde{Y_{2}}\circ\tilde{Y_{2}}.
\end{align*}

Next, we define a Haar measure $dg$ on $G$ according to \cite[Proposition 5.1 in Chapter 5]{mits}. Let
\begin{align*}
k_{\theta}=
\begin{pmatrix}
\mathrm{cos}\frac{\theta}{2} & \mathrm{sin}\frac{\theta}{2} \\
-\mathrm{sin}\frac{\theta}{2} & \mathrm{cos}\frac{\theta}{2}
\end{pmatrix}
,\, a_{s}=
\begin{pmatrix}
e^{\frac{s}{2}} & 0 \\
0 & e^{-\frac{s}{2}}
\end{pmatrix}
,\, n_{x}=
\begin{pmatrix}
1 & x \\
0 & 1
\end{pmatrix}
\end{align*}
for $\theta, s, x\in\mathbb{R}$ and
\begin{align*}
A=\{a_{s}\, |\, s\in\mathbb{R}\}, N=\{n_{x}\, |\, x\in\mathbb{R}\}.
\end{align*}
Then, we get  the Iwasawa decomposition
\begin{align*}
G=KAN\cong K\times A\times N
\end{align*}
and the Cartan decomposition
\begin{align*}
G=KAK.
\end{align*}
We define $dg$ by
\begin{align*}
\int_{G}f(g)dg=\frac{1}{4\pi}\int_{(0,4\pi)\times\mathbb{R}\times\mathbb{R}}f(k_{\theta}a_{s}n_{x})e^{s}d\theta dsdx,\, \, f\in C_{0}(G).
\end{align*}

\subsection{Decomposition of $L^2(SL(2,\mathbb{R}),dg)$}
For $n\in\mathbb{Z}$, let $\tau_{n}:K\rightarrow GL(1,\mathbb{C})\cong \mathbb{C}^{\times}$ be the function defined by $\tau_{n}(k_{\theta})=e^{in\frac{\theta}{2}}$. Then, $\hat{K}=\{\tau_{n}\}_{n\in\mathbb{Z}}$. Note that all the representation space $V_{\tau_{n}}$ is $\mathbb{C}$. According to Section \ref{decom}, we get the following decomposition:
\begin{align*}
L^2(G,dg)=\sum_{n\in\mathbb{Z}}L^2(G,\tau_{n}),
\end{align*}
where $L^2(G,\tau_{n})=\{f\in L^2(G)\, |\, f(gk)=\tau_{n}(k^{-1})f(g),\, \,  g\in G,\, k\in K\}$. In this situation, $L^2(G,\tau_{n},\tau_{n})\subset L^2(G,\tau_{n})$. Let $P_{\tau_{n}}:L^2(G,dg)\rightarrow L^2(G,\tau_{n})$ be the orthogonal projection operator. Then we have
\begin{align*}
P_{\tau_{n}}f(g)=\int_{K}\tau_{n}(k)f(gk)dk
\end{align*}
since
\begin{align*}
\int_{K}\tau_{n}(k')f((gk)k')dk'=\int_{K}\tau_{n}(k^{-1}k')f(gk')dk'=\tau_{n}(k^{-1})\int_{K}\tau_{n}(k')f(gk')dk'
\end{align*}
for $g\in G,\, k\in K$.

\subsection{Irreducible unitary representations of $SL(2,\mathbb{R})$}
First, we describe important irreducible unitary representations according to \cite[Chapter 5]{mits}. For $\varepsilon=0,1$ and $\nu\in\mathbb{R}$, let $U_{\varepsilon, \nu}:G\rightarrow U(L^2(\mathbb{R},\frac{1}{\pi}dx))$ be the unitary representation of $G$ given by
\begin{align*}
U_{\varepsilon, \nu}(g)f(x)=(\mathrm{sgn}(cx+d))^{\varepsilon}(cx+d)^{i2\nu-1}f(\frac{ax+b}{cx+d}),
\end{align*}
where $g^{-1}=
\begin{pmatrix}
a & b\\
c & d
\end{pmatrix}
\in G$. It is known that $U_{\varepsilon, -\nu}\cong U_{\varepsilon, \nu}$. The set $\{U_{\varepsilon, \nu}\}_{\varepsilon=0,1, \nu>0}$ is called the continuous series. We denote by $\mathbb{H}$ the upper half plane in $\mathbb{C}$. For an integer $m\geq 2$, let
\begin{flushleft}
$\mathscr{H}_{m}^{+}=\{f:\mathbb{H}:\rightarrow\mathbb{C}\, |\, f:$\, holomorphic$, ||f||=\int_{\mathbb{H}}|f(z)|^2y^{m-2}dxdy<\infty\}$,\\
$\mathscr{H}_{m}^{-}=\{f:\mathbb{H}:\rightarrow\mathbb{C}\, |\, f:$\, anti-holomorphic$, ||f||=\int_{\mathbb{H}}|f(z)|^2y^{m-2}dxdy<\infty\}$.
\end{flushleft}
These are Hilbert spaces. Let $U_{m}^{+}:G\rightarrow U(\mathscr{H}_{m}^{+})$ be the unitary representation of $G$ given by
\begin{align*}
U_{m}^{+}(g)f(z)=(cz+d)^{-m}f(\frac{az+b}{cz+d})
\end{align*}
and $U_{m}^{-}:G\rightarrow GL(\mathscr{H}_{m}^{-})$ be the unitary representation of $G$ given by
\begin{align*}
U_{m}^{-}(g)f(z)=(c\bar{z}+d)^{-m}f(\frac{az+b}{cz+d}),
\end{align*}
where $g^{-1}=
\begin{pmatrix}
a & b\\
c & d
\end{pmatrix}
\in G$. The set $\{U_{m}^{+}\}_{m\geq 2}\cup\{U_{m}^{-}\}_{m\geq 2}$ is called the regular discrete series. Let $\hat{G}_{p}=\{U_{\varepsilon, \nu}\}_{\varepsilon=0,1, \nu>0}\cup\{U_{m}^{+}\}_{m\geq 2}\cup\{U_{m}^{-}\}_{m\geq 2}$. This set is called the regular principal series.

Next, we discuss characters according to \cite[Section 5.6]{mits}. Let $\mathfrak{g}_{\mathbb{C}}$ be a complexification of $\mathfrak{g}$, $U(\mathfrak{g}_{\mathbb{C}})$ be the enveloping algebra of $\mathfrak{g}_{\mathbb{C}}$, $\mathfrak{Z}\subset U(\mathfrak{g}_{\mathbb{C}})$ be the center of $U(\mathfrak{g}_{\mathbb{C}})$ and $U\in\hat{G}$. For $u\in \mathfrak{Z}$, we see that $dU(u)$ is a scalar operator by the Schur Lemma, so that we can write $dU(u)=\chi_{U}(u)Id_{H_{U}}$ with $\chi_{U}(u)\in\mathbb{C}$. Then, the homomorphism $\mathfrak{Z}\ni u\mapsto\chi_{U}(u)\in\mathbb{C}$ is defined. The function $\chi_{U}$ is called the character of $U$. Since $U(\mathfrak{g}_{\mathbb{C}})$ is isomorphic to the space $\mathbb{D}_{L}(G)$ of left invariant differential operators (see \cite[Section 2.5]{mits}), we can identify $\mathfrak{Z}$ with a center of $\mathbb{D}_{L}(G)$. Let
\begin{align*}
C_{G}=-\tilde{X_{1}}\circ\tilde{X_{1}}+\tilde{Y_{1}}\circ\tilde{Y_{1}}+\tilde{Y_{2}}\circ\tilde{Y_{2}},
\end{align*}
then $C_{G}\in\mathfrak{Z}$ and $\Delta_{G}=2\tilde{X_{1}}\circ\tilde{X_{1}}+C_{G}$. The element $C_{G}$ is called the Casimir operator of $G$. We can describe the value of characters at the center $C_{G}$.
\begin{thm}\label{character}(cf. \cite[Theorem 6.4 in Chapter 5]{mits})
We have
\begin{eqnarray}
\chi_{U_{\varepsilon,\nu}}(C_{G})&=&-\frac{1}{2}(\nu^2+\frac{1}{4}), \nonumber \\
\chi_{U_{m}^{+}}(C_{G})&=&\frac{m(m-2)}{8}, \nonumber \\
\chi_{U_{m}^{-}}(C_{G})&=&\frac{m(m-2)}{8}. \nonumber
\end{eqnarray}
\end{thm}

Finally, we describe the weights and $\hat{G}(\tau_{n})_{p}=\hat{G}(\tau_{n})\cap\hat{G}_{p}$ according to \cite[Section 5.6]{mits}. For $U\in\hat{G}$, let $\Lambda_{U}=\{n\in\mathbb{Z}\, |\, m(\tau_{n},U)\neq 0\}$ which is called the set of weights of $U$. For $n\in\Lambda_{U}$, let $e_{\tau_{n}}^{U}\in H_{U}$ be a normalized weight vector of $n$. Now, we can give the set of weights.
\begin{thm}(cf. \cite[Theorem 6.4 in Chapter 5]{mits})\label{weight}
We have
\begin{eqnarray}
\Lambda_{U_{\varepsilon,\nu}}&=&\{n\in\mathbb{Z}\, |\, n\equiv\varepsilon\, (\mathrm{mod} 2)\}, \nonumber \\
\Lambda_{U_{m}^{+}}&=&\{n\in\mathbb{Z}\, |\, n\leq -m,\, n\equiv m\, (\mathrm{mod} 2)\}, \nonumber \\
\Lambda_{U_{m}^{-}}&=&\{n\in\mathbb{Z}\, |\, n\geq m,\, n\equiv m\, (\mathrm{mod} 2)\}. \nonumber
\end{eqnarray}
\end{thm}
Using Theorem \ref{weight}, we can describe $\hat{G}(\tau_{n})_{p}$.
\begin{cor}\label{intedom}
\begin{eqnarray}
\mathrm{If}\,  n \geq 0, \mathrm{\, we\, \, have}\nonumber \\
\hat{G}(\tau_{n})_{p}&=&\{U_{\varepsilon,\nu}\, |\, \varepsilon\equiv n(\mathrm{mod} 2),\, \nu>0\}\cup\{U_{m}^{+}\, |\, 2\leq m\leq |n|,\, m\equiv n(\mathrm{mod} 2)\}. \nonumber \\
\mathrm{If}\,  n < 0, \mathrm{\, we\, \, have}\nonumber \\
\hat{G}(\tau_{n})_{p}&=&\{U_{\varepsilon,\nu}\, |\, \varepsilon\equiv n(\mathrm{mod} 2),\, \nu>0\}\cup\{U_{m}^{-}\, |\, 2\leq m\leq |n|,\, m\equiv n(\mathrm{mod} 2)\}. \nonumber
\end{eqnarray}
\end{cor}

\subsection{The spherical functions on $SL(2,\mathbb{R})$}\label{SFonSL2R}
First, we discuss derivatives of the function $\Phi_{\tau_{n}}^{U}(g^{-1})$. For $U\in\hat{G}$, let $\langle,\rangle_{H_{U}}$ be an inner product of $H_{U}$. We get $\Phi_{\tau_{n}}^{U}(g)=\langle U(g)e_{\tau_{n}}^{U},e_{\tau_{n}}^{U}\rangle_{H_{U}}$. On the other hand,
\begin{align*}
\Phi_{\tau_{n}}^{U}(g^{-1})=\langle U(g^{-1})e_{\tau_{n}}^{U},e_{\tau_{n}}^{U}\rangle_{H_{U}}=\langle e_{\tau_{n}}^{U},U(g)e_{\tau_{n}}^{U}\rangle_{H_{U}}=\overline{\Phi_{\tau_{n}}^{U}(g)}.
\end{align*}
We put $\eta:G\ni g\mapsto g^{-1}\in G$ and $\tilde{\Phi}_{\tau_{n}}^{U}=\Phi_{\tau_{n}}^{U}\circ\eta$. Let $X\in\mathfrak{g}$ and $\tilde{X}$ be the left invariant vector field on $G$ induced by $X$. Then,
\begin{eqnarray}
\tilde{X}\tilde{\Phi}_{\tau_{n}}^{U}(g)&=&\frac{d}{dt}\langle e_{\tau_{n}}^{U},U(ge^{tX})e_{\tau_{n}}^{U}\rangle_{H_{U}}|_{t=0} \nonumber \\
&=&\langle e_{\tau_{n}}^{U},U(g)dU(X)e_{\tau_{n}}^{U}\rangle_{H_{U}}. \nonumber
\end{eqnarray}
To know $\tilde{X}\tilde{\Phi}_{\tau_{n}}^{U}(g)$ more, we should know an expression of $dU(X)e_{\tau_{n}}^{U}$. We can observe $dU(X)e_{\tau_{n}}^{U}$ as follows.
\begin{thm}\label{derivative}(cf. \cite[Section 6.5, 6.6]{lang})
We have
\begin{eqnarray}
&dU_{\varepsilon, \nu}(X_{1})e_{\tau_{n}}^{U_{\varepsilon, \nu}}=&-\frac{1}{\sqrt{8}}ine_{\tau_{n}}^{U_{\varepsilon, \nu}}, \nonumber \\
&dU_{\varepsilon, \nu}(Y_{1})e_{\tau_{n}}^{U_{\varepsilon, \nu}}=&\frac{1}{2\sqrt{8}}\biggl( (i2\nu +1-n)e_{\tau_{n-2}}^{U_{\varepsilon, \nu}} + (i2\nu +1+n)e_{\tau_{n+2}}^{U_{\varepsilon, \nu}} \biggr) , \nonumber \\
&dU_{\varepsilon, \nu}(Y_{2})e_{\tau_{n}}^{U_{\varepsilon, \nu}}=&\frac{i}{2\sqrt{8}}\biggl( (i2\nu +1-n)e_{\tau_{n-2}}^{U_{\varepsilon, \nu}} - (i2\nu +1+n)e_{\tau_{n+2}}^{U_{\varepsilon, \nu}} \biggr) , \nonumber \\
&dU_{m}^{\pm}(X_{1})e_{\tau_{n}}^{U_{m}^{\pm}}=&-\frac{1}{\sqrt{8}}ine_{\tau_{n}}^{U_{m}^{\pm}}, \nonumber \\
&dU_{m}^{\pm}(Y_{1})e_{\tau_{n}}^{U_{m}^{\pm}}=&\frac{1}{2\sqrt{8}}\biggl( (m-n)e_{\tau_{n-2}}^{U_{m}^{\pm}} + (m+n)e_{\tau_{n+2}}^{U_{m}^{\pm}} \biggr),  \nonumber \\
&dU_{m}^{\pm}(Y_{2})e_{\tau_{n}}^{U_{m}^{\pm}}=&\frac{i}{2\sqrt{8}}\biggl( (m-n)e_{\tau_{n-2}}^{U_{m}^{\pm}} - (m+n)e_{\tau_{n+2}}^{U_{m}^{\pm}} \biggr) . \nonumber
\end{eqnarray}
\end{thm}
By Theorem \ref{character} and Theorem \ref{derivative}, we obtain $\tilde{X_{1}}\circ\tilde{X_{1}}\tilde{\Phi}_{\tau_{n}}^{U}(g)=-\frac{1}{8}n^2\tilde{\Phi}_{\tau_{n}}^{U}(g)$ and $C_{G}\tilde{\Phi}_{\tau_{n}}^{U}(g)=\chi_{U}(C_{G})\tilde{\Phi}_{\tau_{n}}^{U}(g)$, so that
\begin{align*}
\Delta_{G}\tilde{\Phi}_{\tau_{n}}^{U}(g)=\lambda_{\tau_{n}}^{U}\tilde{\Phi}_{\tau_{n}}^{U}(g),\, \, \mathrm{where}\, \,  \lambda_{\tau_{n}}^{U}=-\frac{1}{4}n^2+\chi_{U}(C_{G})<0.\, \, \, (2)
\end{align*}

Next, we see expressions of the functions $\tilde{\Phi}_{\tau_{n}}^{U}(g)$ and its derivatives explicitly. To do so,  it is enough to know the expressions of the functions $\langle e_{\tau_{n_{1}}}^{U},U(g)e_{\tau_{n_{2}}}^{U}\rangle_{H_{U}}$ explicitly. By using the Cartan decomposition $G=KAK$, we have
\begin{eqnarray}
\langle e_{\tau_{n_{1}}}^{U},U(k_{\theta_{1}}a_{s}k_{\theta_{2}})e_{\tau_{n_{2}}}^{U}\rangle_{H_{U}}&=&\langle U(k_{\theta_{1}}^{-1})e_{\tau_{n_{1}}}^{U},U(a_{s})U(k_{\theta_{2}})e_{\tau_{n_{2}}}^{U}\rangle_{H_{U}} \nonumber \\
&=&e^{-i\frac{n_{1}\theta_{1}+n_{2}\theta_{2}}{2}}\langle e_{\tau_{n_{1}}}^{U},U(a_{s})e_{\tau_{n_{2}}}^{U}\rangle_{H_{U}} \nonumber
\end{eqnarray}
for any $\theta_{1}, \theta_{2}\in\mathbb{R}, s>0$. Thus the function $\langle e_{\tau_{n_{1}}}^{U},U(g)e_{\tau_{n_{2}}}^{U}\rangle_{H_{U}}$ is determined by the function $\langle e_{\tau_{n_{1}}}^{U},U(a_{s})e_{\tau_{n_{2}}}^{U}\rangle_{H_{U}}$. Explicit expression of $\langle e_{\tau_{n_{1}}}^{U},U(a_{s})e_{\tau_{n_{2}}}^{U}\rangle_{H_{U}}$ is given by two ways. We refer to both of them.
\begin{thm}\label{hypergeo}(cf. \cite[Proposition 7.16 in Chapter 5]{mits})
Let $\nu$ be a complex number such that $\chi_{U}(-2C)=\nu^2+1$, $a=i\nu+\frac{1}{2}+\frac{|n_{1}-n_{2}|}{4}+\frac{n_{1}+n_{2}}{4}$ and $b=i\nu+\frac{1}{2}+\frac{|n_{1}-n_{2}|}{4}-\frac{n_{1}+n_{2}}{4}$. Then
\begin{align*}
\langle e_{\tau_{n_{1}}}^{U},U(a_{s})e_{\tau_{n_{2}}}^{U}\rangle_{H_{U}}=\Bigl( \mathrm{tanh}\frac{s}{2} \Bigr) ^{\frac{|n_{1}-n_{2}|}{2}}\Bigl( \mathrm{cosh}\frac{s}{2}\Bigr)^{-i2\nu-1}F(a,b,1,(\mathrm{tanh}\frac{s}{2})^2)
\end{align*}
where $F$ is the hypergeometric function.
\end{thm}
\begin{thm}\label{mtxelementSL2R}(cf. \cite[(2.15)]{koo2})
The function $\langle U(a_{s})e_{\tau_{n_{1}}}^{U},e_{\tau_{n_{2}}}^{U}\rangle_{H_{U}}$ is given by follows.
\begin{eqnarray}
&\langle U_{\varepsilon,\nu}(a_{s})e_{\tau_{n_{1}}}^{U_{\varepsilon,\nu}},e_{\tau_{n_{2}}}^{U_{\varepsilon,\nu}}\rangle_{H_{U_{\varepsilon,\nu}}}=\Bigl( \mathrm{cosh}\frac{s}{2}\Bigr)^{-i2\nu-1}\nonumber \\
&\frac{1}{4\pi}\int_{0}^{4\pi}\Bigl( 1-\mathrm{tanh}\frac{s}{2}\, e^{i\psi} \Bigr)^{-i\nu+\frac{n_1}{2}-\frac{1}{2}}\Bigl(1- \mathrm{tanh}\frac{s}{2}\, e^{-i\psi} \Bigr)^{-i\nu-\frac{n_1}{2}-\frac{1}{2}}e^{i\frac{n_2-n_1}{2}\psi}d\psi. \nonumber \\
&\langle U_{m}^{\pm}(a_{s})e_{\tau_{n_{1}}}^{U_{m}^{\pm}},e_{\tau_{n_{2}}}^{U_{m}^{\pm}}\rangle_{H_{U_{m}^{\pm}}}=\Bigl( \mathrm{cosh}\frac{s}{2}\Bigr)^{-m}\nonumber \\
&\frac{1}{4\pi}\int_{0}^{4\pi}\Bigl( 1-\mathrm{tanh}\frac{s}{2}\, e^{i\psi} \Bigr)^{\frac{-m+n_1}{2}}\Bigl(1- \mathrm{tanh}\frac{s}{2}\, e^{-i\psi} \Bigr)^{\frac{-m-n_1}{2}}e^{i\frac{n_2-n_1}{2}\psi}d\psi. \nonumber
\end{eqnarray}
\end{thm}
Thus, we can calculate $\tilde{\Phi}_{\tau_{n}}^{U}(g)=\Phi_{\tau_{n}}^{U}(g^{-1})$ and its derivatives explicitly. Moreover, we discuss an another explicit form of the spherical function $\Phi_{\tau_{n}}^{U}$. Let $\alpha, \beta, \lambda\in\mathbb{C}$ $(\alpha\neq -1, -2, ...)$. We give a differential equation on $\mathbb{R}$:
\begin{align*}
\frac{d^2\phi}{dt^2}+\Bigl( (2\alpha+1)\frac{1}{\mathrm{tanh}t}+(2\beta+1)\mathrm{tanh}t \Bigr)\frac{d\phi}{dt}+\Bigl( \lambda^2 +(\alpha+\beta+1)^2 \Bigr)\phi =0.
\end{align*}
When we assume $\phi(0)=1$ and $\phi$ is even, the unique solution is called a Jacobi function. We denote $\phi=\phi_{\lambda}^{(\alpha,\beta)}$. We can write $\Phi_{\tau_{n}}^{U}$ by using a Jacobi function (see \cite[Theorem 2.1]{koo2}):
\begin{eqnarray}
\Phi_{\tau_{n}}^{U_{\varepsilon,\nu}}(a_{s})&=& (\mathrm{cosh}s)^{n}\phi_{2\nu}^{(0,n)}(\frac{s}{2}), \nonumber \\
\Phi_{\tau_{n}}^{U_{m}^{\pm}}(a_{s})&=&(\mathrm{cosh}s)^{n}\phi_{i(m-1)}^{(0,n)}(\frac{s}{2}).\nonumber
\end{eqnarray}
These are expressions of $\Phi_{\tau_{n}}^{U}$ using Jacobi functions.

Finally, we discuss the Plancherel measure $d\mu$ on $\hat{G}$. It is known that $d\mu$ is described by the following formulas (see \cite[Section 8.4]{lang}, \cite[Section 7.2.1]{wa2}):
\begin{eqnarray}
d\mu(U_{0,\nu})&=&\frac{1}{2\pi}\nu\, \mathrm{tanh}\pi\nu \, d\nu, \nonumber \\
d\mu(U_{1,\nu})&=&\frac{1}{2\pi}\nu\, \frac{1}{\mathrm{tanh}\pi\nu} \, d\nu, \nonumber \\
d\mu(U_{m}^{\pm})&=&\frac{m-1}{4\pi}, \nonumber \\
d\mu(\hat{G}\setminus\hat{G}_{p})&=&0.\, \, \, \, \, \, \, \,\, \, \, \, \, \, \, \, \, \, \, \, \, \, \, \, \, \, \,\,\,\,\,\,\,\, (3) \nonumber
\end{eqnarray}
\begin{rem}
The Plancherel measure $d\mu$ is determined uniquely by the Haar measure $dg$ (cf. \cite[Theorem 7.2.1.1]{wa2}). Of course, an expression of $d\mu$ depends on a parametrization of $\hat{G}$. We describe $d\mu$ following \cite[(8.3), (8,4) in Chapter 5]{mits}.
\end{rem}

\subsection{Principle series representations of $SL(2,\mathbb{R})$}
We apply the argument of Section \ref{Principle series} to $SL(2,\mathbb{R})$. We can show $M=\{\pm I_{2}\}$. Let $Q=MAN$. Since each parabolic subgroup $Q'$ is a block upper triangular subgroup (see \cite[Section 5.5]{knapp2}), we get $Q'=Q$ or $G$. When $Q'=Q$, we have $M'=M$, $A'=A$ and $N'=N$. When $Q'=G$, we have $M'=G$ and both $A'$ and $N'$ are trivial. In addition, both $Q$ and $G$
 are cuspidal. Next, we describe the result of induced representations when $Q'=Q$. We denote the set  of irreducible unitary representations  of $M$ by $\{\sigma_{0},\sigma_{1}\}$, where $\sigma_{0}$ is the trivial representation and $\sigma_{1}$ is defined by $\sigma_{1}(\pm I_{2})=\pm 1$.
\begin{thm}(cf. \cite[Proposition 2.10 in Chapter 5]{mits})
For each $\varepsilon=0,1$ and $\nu>0$, $U_{\varepsilon,\nu}\cong U(Q,\sigma_{\varepsilon},i\nu)$.
\end{thm} 

\subsection{The Helgason-Fourier transforms on $SL(2,\mathbb{R})$}
We rewrite Theorem \ref{Plancherel_formula} when $G=SL(2,\mathbb{R})$. We fix $n\in\mathbb{Z}$. First, we consider the case when $Q'=Q$. Since $A'=A$, the group $A_{1}$ is trivial. Thus all $\mu_{1}$ is trivial. Since $M'=M$, we have $\sigma'=\tilde{\sigma}'$. In addition, there is at most one $\sigma$ such that $m(\tau_{n},U(Q,\sigma,\nu))\neq0$ and $T_{\tilde{\sigma}}$ is an identity map. We denote this $\sigma$ by $\sigma_{\varepsilon}$. Next, we consider the case when $Q'=G$. Since $A'$ is trivial, all $\nu'$ is trivial. $T_{\tilde{\sigma}'}$ is also an identity map. In addition, the sum $\sum_{\sigma'}$ is a finite sum (recall Corollary \ref{intedom}). In conclusion, we get the following formula. Let $f_{1},f_{2}\in C_{0}^{\infty}(G,\tau)$. Then,
\begin{flushleft}
$\langle f_{1},f_{2}\rangle_{L^2(G,dg)}$\\
$\displaystyle =c_{Q}\int_{\mathfrak{a}^{*}\times K}\langle \tilde{f}_{1}(\nu,k),\tilde{f}_{2}(\nu,k)\rangle_{V_{\tau_{n}}}p_{\sigma_{\varepsilon}}(\nu)d\nu dk$\\
$\displaystyle +c_{G}\sum_{\sigma'_{m}}\int_{K}\langle \tilde{f}_{1}(i\mu_{1},k),\tilde{f}_{2}(-i\mu_{1},k)\rangle_{V_{\tau_{n}}}p_{\sigma'_{m}}dk$,
\end{flushleft}
where $p_{\sigma_{\varepsilon}}(\nu)d\nu=d\mu(U_{\varepsilon,\nu})$, $p_{\sigma'_{m}}=d\mu(U_{m}^{\pm})$ and each $\sigma'_{m}$ is discrete series of $G$ such that $\sigma'_{m}\cong U_{m}^{\pm}\in \hat{G}(\tau_{n})_{p}$. We call this formula is the Plancherel formula for $SL(2,\mathbb{R})$.

\subsection{Calculation of the heat kernel on $SL(2,\mathbb{R})$}
Let $\rho_{G}(t,g)$ be the heat kernel on $G=SL(2,\mathbb{R})$. Then, $\rho_{G}=\sum_{\mathbb{Z}}P_{\tau_{n}}\rho_{G}$. To calculate $\rho_{G}$, it is enough to calculate $P_{\tau_{n}}\rho_{G}$ for each $n\in\mathbb{Z}$. We fix $n\in\mathbb{Z}$. Let $\rho_{t,n}(g)=\int_{\hat{G}(\tau_{n})}\tilde{\Phi}_{\tau_{n}}^{U}(g)e^{t\lambda_{\tau_{n}}^{U}}d\mu(U)$. If $\rho_{t,n}$ satisfies conditions below:
\begin{eqnarray}
&\mathrm{(i)}& \frac{\partial}{\partial t}(\rho_{t,n} * f)=\Delta_{G} (\rho_{t,n} * f),\nonumber \\
&\mathrm{(ii)}& \parallel \rho_{t,n} * f - f \parallel _{L^2(G, dg)} \rightarrow 0 \, \, (t \rightarrow +0),\, \,  f \in C_{0}^{\infty}(G)\cap L^2(G,\tau_{n}), \nonumber
\end{eqnarray}
then $\rho_{t,n}(g)=P_{\tau_{n}}\rho_{G}(t,g)$ by uniqueness of the heat kernel. Thus, to calculate $P_{\tau_{n}}\rho_{G}$, it is enough to show conditions above. First, we show $\mathrm{(i)}$.
\begin{lem}\label{each-conti}
The function $\rho_{t,n}(g)$ is continuous on $(0,\infty)\times G$. 
\end{lem}
\begin{proof}
For $T>0$, we prove that the function $e^{T\lambda_{\tau_{n}}^{U}}$ is a dominant function of the function $\tilde{\Phi}_{\tau_{n}}^{U}(g)e^{t\lambda_{\tau_{n}}^{U}}$ on $[T,\infty)\times G$. By Cauchy-Schwartz's inequality, we have $|\tilde{\Phi}_{\tau_{n}}^{U}(g)|\leq 1$. Thus $|\tilde{\Phi}_{\tau_{n}}^{U}(g)e^{t\lambda_{\tau_{n}}^{U}}|\leq e^{T\lambda_{\tau_{n}}^{U}}$ since $\lambda_{\tau_{n}}^{U}<0$. We calculate an upper bound of $\parallel e^{T\lambda_{\tau_{n}}^{U}} \parallel_{L^1(\hat{G}(\tau_{n}),d\mu)}$. We have
\begin{align*}
\parallel e^{T\lambda_{\tau_{n}}^{U}} \parallel_{L^1(\hat{G}(\tau_{n})_{p},d\mu)}=\int_{\hat{G}(\tau_{n})}e^{T\lambda_{\tau_{n}}^{U}}d\mu(U)=\int_{\hat{G}(\tau_{n})_{p}\cap\{U_{\varepsilon,\nu}\}}+\int_{\hat{G}(\tau_{n})_{p}\cap\{U_{m}^{\pm}\}}.\, \, \mathrm{(4)}
\end{align*}
The first term in (4) is evaluated as
\begin{eqnarray}
\int_{\hat{G}(\tau_{n})_{p}\cap\{U_{\varepsilon,\nu}\}}&\leq&\int_{0}^{\infty}e^{T(-\frac{1}{4}n^2-\frac{1}{2}(\nu^2+\frac{1}{4}))}\frac{1}{2\pi}\nu\, \mathrm{max}\{\mathrm{tanh}\,\nu,\frac{1}{\mathrm{tanh}\,\nu}\}d\nu \nonumber \\
&\leq&\frac{e^{-\frac{1}{8}T-\frac{1}{4}Tn^2}}{\pi}\biggl(\int_{0}^{1}d\nu+\int_{1}^{\infty}e^{-\frac{1}{2}T\nu^2}\nu d\nu \biggr)\nonumber \\
&=&\frac{e^{-\frac{1}{8}T}}{\pi}\biggl(1+\frac{e^{-\frac{1}{2}T}}{T}\biggr)e^{-\frac{1}{4}Tn^2}. \nonumber
\end{eqnarray}
On the other hand, the second term in (4) is observed as
\begin{eqnarray}
\int_{\hat{G}(\tau_{n})_{p}\cap\{U_{m}^{\pm}\}}&=&\sum_{2\leq m\leq |n|, m\equiv n (\mathrm{mod} 2)}e^{T(-\frac{1}{4}n^2+\frac{m(m-2)}{8})}\frac{m-1}{4\pi} \nonumber \\
&\leq&e^{T(-\frac{1}{4}n^2+\frac{1}{8}n^2)}\frac{|n|}{4\pi}\sum_{2\leq m\leq |n|, m\equiv n (\mathrm{mod} 2)} \nonumber \\
&\leq&\frac{1}{8\pi}e^{-\frac{1}{8}Tn^2}n^2. \nonumber
\end{eqnarray}
Therefore the function $e^{T\lambda_{\tau_{n}}^{U}}$ is a dominant function.
\end{proof}
\begin{lem}\label{L2}
For $t>0$, we have $\rho_{t,n}\in L^2(G,dg)$.
\end{lem}
\begin{proof}
By the Theorem \ref{plancherel}, we have
\begin{align*}
\parallel \rho_{t,n} \parallel_{L^2(G,dg)}=\parallel e^{t\lambda_{\tau_{n}}^{U}} \parallel_{L^2(\hat{G}(\tau_{n}),d\mu)}.
\end{align*}
By the same way as the proof of Lemma \ref{each-conti}, we obtain
\begin{align*}
\parallel e^{t\lambda_{\tau_{n}}^{U}} \parallel_{L^2(\hat{G}(\tau_{n}),d\mu)}^2\leq\frac{e^{-\frac{1}{4}t}}{\pi}\biggl(1+\frac{e^{-t}}{2t}\biggr)e^{-\frac{1}{2}tn^2}+\frac{1}{8\pi}e^{-\frac{1}{4}tn^2}n^2.
\end{align*}
\end{proof}
\begin{lem}
For $X\in\mathfrak{g}$, We have
\begin{eqnarray}
\frac{\partial \rho_{t,n}}{\partial t}(g)&=&\int_{\hat{G}(\tau_{n})}\lambda_{\tau_{n}}^{U}\tilde{\Phi}_{\tau_{n}}^{U}(g)e^{t\lambda_{\tau_{n}}^{U}}d\mu(U), \nonumber \\
\tilde{X}\rho_{t,n}(g)&=&\int_{\hat{G}(\tau_{n})}\tilde{X}\tilde{\Phi}_{\tau_{n}}^{U}(g)e^{t\lambda_{\tau_{n}}^{U}}d\mu(U), \nonumber \\
\tilde{X}\circ\tilde{X}\rho_{t,n}(g)&=&\int_{\hat{G}(\tau_{n})}\tilde{X}\circ\tilde{X}\tilde{\Phi}_{\tau_{n}}^{U}(g)e^{t\lambda_{\tau_{n}}^{U}}d\mu(U). \nonumber
\end{eqnarray}
\end{lem}
\begin{proof}
We see that $\rho_{t,n}(g)$ is smooth in $t>0$ and $g\in G$ by using interchanging differentiation and integration. To see the situation, we prove that $\rho_{t,n}$ can be differentiated by $\tilde{Y}_{1}$. For each $T>0$, we find a dominant function of $\tilde{Y}_{1}\tilde{\Phi}_{\tau_{n}}^{U}(g)e^{t\lambda_{\tau_{n}}^{U}}$ on $[T,\infty)\times G$. By Theorem \ref{derivative}, we compute
\begin{eqnarray}
\tilde{Y}_{1}\tilde{\Phi}_{\tau_{n}}^{U_{\varepsilon, \nu}}(g)=\frac{1}{2\sqrt{8}}\biggl(&(-i2\nu +1-n)\langle e_{\tau_{n}}^{U_{\varepsilon, \nu}},U_{\varepsilon,\nu}(g)e_{\tau_{n-2}}^{U_{\varepsilon, \nu}}\rangle_{H_{U_{\varepsilon,\nu}}}&\nonumber \\
&+ (-i2\nu +1+n)\langle e_{\tau_{n}}^{U_{\varepsilon, \nu}},U_{\varepsilon,\nu}(g)e_{\tau_{n+2}}^{U_{\varepsilon, \nu}}\rangle_{H_{U_{\varepsilon,\nu}}}& \biggr) \nonumber
\end{eqnarray}
and
\begin{eqnarray}
\tilde{Y}_{1}\tilde{\Phi}_{\tau_{n}}^{U_{m}^{\pm}}(g)=\frac{1}{2\sqrt{8}}\biggl( &(m-n)\langle e_{\tau_{n}}^{U_{m}^{\pm}},U_{m}^{\pm}(g)e_{\tau_{n-2}}^{U_{m}^{\pm}}\rangle_{H_{U_{m}^{\pm}}}&\nonumber \\
&+ (m+n)\langle e_{\tau_{n}}^{U_{m}^{\pm}},U_{m}^{\pm}(g)e_{\tau_{n+2}}^{U_{m}^{\pm}}\rangle_{H_{U_{m}^{\pm}}}& \biggr). \nonumber
\end{eqnarray}
By Cauchy-Schwartz's inequality, we have
\begin{align*}
|\langle e_{\tau_{n}}^{U},U(g)e_{\tau_{n-2}}^{U}\rangle_{H_{U}}|\leq 1\, \,  \mbox{and}\,\, |\langle e_{\tau_{n}}^{U},U(g)e_{\tau_{n+2}}^{U}\rangle_{H_{U}}|\leq 1.
\end{align*}
Let $\varphi_{n}$ be a function on $\hat{G}(\tau_{n})$ defined by
\begin{eqnarray}
\varphi_{n}(U)=
\begin{cases}
|-i2\nu +1-n|+ |-i2\nu +1+n| & (U=U_{\varepsilon,\nu}) \\
|m-n|+|m+n| & (U=U_{m}^{\pm}).
\end{cases} \nonumber
\end{eqnarray}
Then, we can show that $\varphi_{n}(U)e^{T\lambda_{\tau_{n}}^{U}}$ is a dominant function of $\tilde{Y}_{1}\tilde{\Phi}_{\tau_{n}}^{U}(g)e^{t\lambda_{\tau_{n}}^{U}}$ by using a similarly method in the proof of Lemma \ref{each-conti}. Thus, $\rho_{t,n}$ can be differentiated by $\tilde{Y}_{1}$ and
\begin{align*}
\tilde{Y}_{1}\rho_{t,n}(g)=\int_{\hat{G}(\tau_{n})}\tilde{Y}_{1}\tilde{\Phi}_{\tau_{n}}^{U}(g)e^{t\lambda_{\tau_{n}}^{U}}d\mu(U).
\end{align*}
\end{proof}
\begin{lem}\label{conv-each-conti}
The function $\rho_{t,n}*f$ is continuous on $(0,\infty)\times G$.
\end{lem}
\begin{proof}
For $T>0$, we prove that the function $e^{T\lambda_{\tau_{n}}^{U}}|f(h)|$ is a dominant function of the function $\tilde{\Phi}_{\tau_{n}}^{U}(h^{-1}g)e^{t\lambda_{\tau_{n}}^{U}}f(h)$ with respect to $(t,g)\in [T,\infty)\times G$, where the integration domain is $(h,U)\in G\times \hat{G}(\tau_{n})$. Since $|\tilde{\Phi}_{\tau_{n}}^{U}(h^{-1}g)|\leq 1$, $|\tilde{\Phi}_{\tau_{n}}^{U}(h^{-1}g)e^{t\lambda_{\tau_{n}}^{U}}f(h)|\leq e^{T\lambda_{\tau_{n}}^{U}}|f(h)|$. We calculate an upper bound of $\parallel e^{T\lambda_{\tau_{n}}^{U}}|f(h)|\parallel_{L^1(G\times\hat{G}(\tau_{n}),dhd\mu)}$. Since variables are separated, we have
\begin{align*}
\parallel e^{T\lambda_{\tau_{n}}^{U}}|f(h)|\parallel_{L^1(G\times\hat{G}(\tau_{n}),dhd\mu)}=\parallel e^{T\lambda_{\tau_{n}}^{U}}\parallel_{L^1(\hat{G},d\mu)}\parallel f(h)\parallel_{L^1(G,dh)}.
\end{align*}
An upper bound of $\parallel e^{T\lambda_{\tau_{n}}^{U}}\parallel_{L^1(\hat{G},d\mu)}$ is calculated in the proof of Lemma \ref{each-conti}.
\end{proof}
\begin{lem}\label{conv-L2}
We have $\rho_{t,n}*f\in L^2(G,dg)$.
\end{lem}
\begin{proof}
By Theorem \ref{Plancherel_formula} and Theorem \ref{conv_mult},
\begin{flushleft}
$\parallel \rho_{t,n}*f \parallel_{L^2(G,dg)}^2$\\
$\displaystyle =c_{Q}\int_{\mathfrak{a}^{*}\times K}e^{2t\lambda_{\tau_{n}}^{U_{\varepsilon,\nu}}}\parallel \tilde{f}(\nu,k)\parallel^2_{V_{\tau_{n}}}p_{\sigma_{\varepsilon}}(\nu)d\nu dk$\\
$\displaystyle +c_{G}\sum_{\sigma'_{m}}e^{2t\lambda_{\tau_{n}}^{U_{m}^{\pm}}}\int_{K}\langle \tilde{f}(i\mu_{1},k),\tilde{f}(-i\mu_{1},k)\rangle_{V_{\tau_{n}}}p_{\sigma'_{m}}dk$.
\end{flushleft}
The right side hand converges since $\lambda_{\tau_{n}}^{U}<0$.
\end{proof}
\begin{lem}
For $X\in\mathfrak{g}$, we have
\begin{eqnarray}
\frac{\partial (\rho_{t,n}*f)}{\partial t}(g)&=&\biggl(\frac{\partial \rho_{t,n}}{\partial t}\biggr)*f(g), \nonumber \\
\tilde{X}(\rho_{t,n}*f)(g)&=&(\tilde{X}\rho_{t,n})*f(g) \nonumber \\
\mathrm{and\, \, }\tilde{X}\circ\tilde{X}(\rho_{t,n}*f)(g)&=&(\tilde{X}\circ\tilde{X}\rho_{t,n})*f(g). \nonumber
\end{eqnarray}
\end{lem}
The proof of this lemma is almost the same as the proof of Lemma \ref{conv-each-conti}. Then, $\mathrm{(i)}$ is proved. Finally, we show $\mathrm{(ii)}$.
\begin{lem}
$\parallel \rho_{t,n}*f-f\parallel_{L^2(G,dg)}\rightarrow 0\, \, (t\rightarrow +0)$.
\end{lem}
\begin{proof}
By Theorem \ref{Plancherel_formula} and Theorem \ref{conv_mult},
\begin{flushleft}
$\parallel \rho_{t,n}*f - f\parallel_{L^2(G,dg)}^2$\\
$\displaystyle =c_{Q}\int_{\mathfrak{a}^{*}\times K}(1-e^{t\lambda_{\tau_{n}}^{U_{\varepsilon,\nu}}})^2\parallel \tilde{f}(\nu,k)\parallel^2_{V_{\tau_{n}}}p_{\sigma_{\varepsilon}}(\nu)d\nu dk$\\
$\displaystyle +c_{G}\sum_{\sigma'_{m}}(1-e^{t\lambda_{\tau_{n}}^{U_{m}^{\pm}}})^2\int_{K}(1-e^{t\lambda_{\tau_{n}}^{U_{m}^{\pm}}})^2\langle \tilde{f}_{1}(i\mu_{1},k),\tilde{f}_{2}(-i\mu_{1},k)\rangle_{V_{\tau_{n}}}p_{\sigma'_{m}}dk$.
\end{flushleft}
The first term converges to $0$ by the monotone convergence theorem. Other terms converge to $0$ clearly.
\end{proof}
Then, $\mathrm{(ii)}$ is proved. Thus, we get the main result.
\begin{thm}\label{MyThm1}
Let $G=SL(2,\mathbb{R})$. The function
\begin{align*}
\rho(t,g)=\sum_{n\in\mathbb{Z}}\int_{\hat{G}(\tau_{n})}\tilde{\Phi}_{\tau_{n}}^{U}(g)e^{t\lambda_{\tau_{n}}^{U}}d\mu(U),\, \, (t,g)\in(0,\infty)\times G
\end{align*}
is the heat kernel on $G$.
\end{thm}
In addition, we can describe $\rho(t,g)$ explicitly by Corollary \ref{intedom}, Theorem \ref{hypergeo} (or Theorem \ref{mtxelementSL2R}), $(2)$ and $(3)$.

\section{Future problems}\label{futher}
In this paper, we have calculated the heat kernel on $SL(2,\mathbb{R})$. Our next problem is Problem \ref{SL2Cheat} in Section 2 about the heat kernel on $SL(2,\mathbb{C})$. We have discussed in a general situation in Section 3-6. Thus, if $G$ is a non-compact semisimple Lie group $G$ having a multiplicity-free subgroup $K$, the heat kernel on $G$ will be calculated by a similar way.

Furthermore, there are problems about the heat kernel related to the Segal-Bargmann space. We introduce them. Let $G$ be a Lie group and $\rho_{G}$ be the heat kernel on $G$. First, we discuss problems when $G=SL(2,\mathbb{C})$.
\begin{prob}(\cite[Open problem 1]{bh3})
Let $t>0$. Prove that the set of linear sums of matrix entries of finite dimensional holomorphic representations is a dense subspace of $HL^2(SL(2,\mathbb{C}),\rho_{SL(2,\mathbb{C})}(t,g)dg)$.
\end{prob}
\begin{prob}(\cite[Open problem 2]{bh3})
Let $t,\varepsilon>0$. Prove that the space \\$HL^2(SL(2,\mathbb{C}),\rho_{SL(2,\mathbb{C})}(t+\varepsilon,g)dg)$ is a dense subspace of the space \\$HL^2(SL(2,\mathbb{C}),\rho_{SL(2,\mathbb{C})}(t,g)dg)$.
\end{prob}
These problems seems not yet solved. Generalizing them, we can consider other problems as follows.
\begin{prob}
Let $t>0$. Prove that the set of linear sums of matrix entries of finite dimensional representations is a dense subspace of $L^2(SL(2,\mathbb{R}),\rho_{SL(2,\mathbb{R})}(t,g)dg)$.
\end{prob}
\begin{prob}
Let $t,\varepsilon>0$. Prove that the space \\
$L^2(SL(2,\mathbb{R}),\rho_{SL(2,\mathbb{R})}(t+\varepsilon,g)dg)$ is a dense subspace of the space\\$L^2(SL(2,\mathbb{R}),\rho_{SL(2,\mathbb{R})}(t,g)dg)$.
\end{prob}
If we know an explicit expression of the heat kernel, we may get a hint to approach these problems.


\begin{thebibliography}{99}
\bibitem{b1} V. Bargmann, On a Hilbert spaces of analytic functions and an associated integral transform, Comm. Pure Appl. Math. ${\bold 14}$(1961), 187-214.
\bibitem{ca1} R. Camporesi, The spherical transform for homogeneous vector bundles over Riemannian symmetric spaces, J. Lie Theory ${\bold 7}$ (1997), no. 1, 29-60.
\bibitem{ca2} R. Camporesi, The Helgason Fourier transform for homogeneous vector bundles over Riemannian symmetric spaces, Pacfic. J. Math. ${\bold 179}$ (1997), no. 2, 263-300. 
\bibitem{cha} I. Chavel, Eigenvalues in Riemannian geometry, Academic Press, 1984.
\bibitem{dod} J. Dodziuk, Maximum principle for parabolic inequalities and the heat flow on open manifolds, Indiana Univ. Math. J. ${\bold 32}$ (1983), no. 5, 703-716.
\bibitem{gan} R. Gangolli, Asymptotic behavior of spectra of compact quotients of certain symmetric spaces, Acta. Math., ${\bold 121}$ (1968), 151-192.
\bibitem{bh1} B. Hall, The Segal-Bargmann ``coherent state'' transform for compact Lie groups, J. Funct. Anal. ${\bold 122}$ (1994), no. 1, 103-151.
\bibitem{bh2} B. Hall, Holomorphic methods in analysis and mathematical physics. In, ``First Summer School in Analysis and Mathematical Physics'' (S.Perez-Esteva and C. Vilegas-Blas, Eds.), 1-59, Contemp. Math. 260, Amer. Math. Soc., 2000.
\bibitem{bh3} B. Hall, Harmonic analysis with respect to heat kernel measure, Bull. Amer. Math. Soc. (N.S.) ${\bold 38}$ (2001), no. 1, 43-78.
\bibitem{helg1} S. Helgason, Differential geometry, Lie groups, and symmetric spaces, Academic Press, 1978.
\bibitem{helg2} S. Helgason, Groups and geometric analysis, Academic Press, 1984.
\bibitem{helg3} S. Helgason, Geometric analysis on symmetric spaces, American Mathematical Society, 1994.
\bibitem{hoc} G. Hochschild, The structure of Lie groups, Holden-day, 1965.
\bibitem{knapp} A. W. Knapp, Lie groups beyond an introduction, Second edition, Birkh$\ddot{a}$user Boston, 2002.
\bibitem{knapp2} A. W. Knapp, Representation theory of semisimple groups, Princeton University Press, 1986.
\bibitem{koo} T. H. Koornwinder, A note on the multiplicity free reduction of certain orthogonal and unitary groups, Indag. Math. {\bf 44} (1982), 215-218.
\bibitem{koo2} T. H. Koornwinder, The representation theory of $SL(2,\mathbb{R})$, a non-infinitesimal approach, Enseign. Math. (2) ${\bold 28}$ (1982), 53-90.
\bibitem{lang} S. Lang, SL$_{2}$(R), Springer Verlag, 1975.
\bibitem{mck} H. P. Mckean, An upper bound to the spectrum of $\Delta$ on a manifold of negative curvature, J. Differential geometry, 4 (1970), 359-366.
\bibitem{mum} D. Mumford, Tata Lectures on Theta I, Birkh$\ddot{a}$user Boston, 1983.
\bibitem{neil} B. O'Neill, Semi-Riemannian geometry, Academic Press, 1983.
\bibitem{ne1} E. Nelson, analytic vectors, Annals of mathematics 70 (1959), 572-615.
\bibitem{s1} I. E. Segal, The complex-wave representation of the free Boson field, "Topics is Functional Analysis", Advances in Mathematics Supplementary Studies, Vol. 3, Academic Press, 1978.
\bibitem{st} E. M. Stein, Topics in Harmonic-Analysis Related to the Littlewood-Paley Theory, Annals of mathematics studies princeton university press, 1970.
\bibitem{mits} M. Sugiura, Unitary representations and harmonic analysis, North-Holland, 1990.
\bibitem{ura} H. Urakawa, On the least positive eigenvalue of the Laplacian for compact group manifolds, J. Math. Soc. Japan 31 (1979), 209-226.
\bibitem{wa1}G. Warner, Harmonic Analysis on Semi-Simple Lie Groups I, Springer Verlag, 1972.
\bibitem{wa2}G. Warner, Harmonic Analysis on Semi-Simple Lie Groups II, Springer Verlag, 1972.
\end{thebibliography}
\end{document}